\documentclass[10pt,twocolumn]{IEEEtran}

\usepackage{amsmath,amssymb,amsthm}
\usepackage[utf8x]{inputenc}
\usepackage{graphicx}
\usepackage{indentfirst}
\usepackage{cmap}
\usepackage{ifthen}
\usepackage{tikz}
\usepackage{algorithm}
\usepackage{algorithmic}
\usepackage{pifont}
\usepackage{url}
\usepackage[normalem]{ulem}

\usepackage{algorithm,algorithmic}

\newcommand{\Sum}{\displaystyle\sum\limits}

\newcommand{\ceil}[1]{\left\lceil{#1}\right\rceil}

\newcommand{\RR}{\mathbb{R}}

\newcommand{\eps}{\varepsilon}

\newcommand{\ol}{\overline}
\renewcommand{\le}{\leqslant}
\renewcommand{\ge}{\geqslant}
\renewcommand{\hat}{\widehat}

\newtheorem{theorem}{Theorem}[section]
\newtheorem{corollary}{Corollary}[theorem]
\newtheorem{lemma}[theorem]{Lemma}
\newtheorem{assumption}[theorem]{Assumption}
\newtheorem{definition}[theorem]{Definition}
\newtheorem{remark}[theorem]{Remark}
\newtheorem{proposition}[theorem]{Proposition}

\newcommand{\circledOne}{\text{\ding{172}}}
\newcommand{\circledTwo}{\text{\ding{173}}}

\newcommand\numberthis{\addtocounter{equation}{1}\tag{\theequation}}

\DeclareMathOperator*{\argmax}{arg\,max}

\DeclareMathOperator{\kernel}{Ker}

\def\changes#1{{\color{black}#1}} 
\def\changesr#1{{\color{black}#1}} 
\def\changesm#1{{\color{black}#1}} 
\def\changesbr#1{{\color{black}#1}} 

\usepackage{pgfplotstable}
\usepackage{pgfplots}
\pgfplotsset{
	table/search path={plot_figures},
}
\usepackage{tikz}
\usepackage{xr}
\usetikzlibrary{external}
\usetikzlibrary{arrows,%
	petri,%
	topaths}%
\usetikzlibrary{graphs,graphs.standard}
\usepackage{tkz-berge}
\graphicspath{{Figures/}}
\usepackage{subfigure}
\pgfdeclarelayer{bg}    
\pgfsetlayers{bg,main}  
\pgfplotsset{compat=1.14}
\usepackage{cite}

\allowdisplaybreaks

\begin{document}
	
	\title{Optimal distributed convex optimization on slowly time-varying graphs}
	\author{Alexander~Rogozin$^{*}$~\IEEEmembership{}
		C\'esar~A.~Uribe$^{*}$~\IEEEmembership{}
		Alexander Gasnikov~\IEEEmembership{}
		Nikolay Malkovsky~\IEEEmembership{}
		Angelia Nedi\'{c}~\IEEEmembership{}
		\thanks{A. Rogozin (\textit{aleksandr.rogozin@phystech.edu}) is with Moscow Institute of Physics and Technology. C.A. Uribe (\textit{cauribe@mit.edu}) is with the Laboratory for Information and Decision Systems (LIDS),
			Massachusetts Institute of Technology. A. Gasnikov (\textit{gasnikov@yandex.ru}) is with the Moscow Institute of Physics and Technology, and the Institute for Information Transmission Problems RAS. N. Malkovsky (\textit{malkovskynv@gmail.com}) is with National Research University Higher School of Economics, Russia. A. Nedi\'{c} (\textit{angelia.nedich@asu.edu}) is with the \changesr{School of Electrical, Computer and Energy Engineering}, Arizona State University, and Moscow Institute of Physics and Technology.}%
		\thanks{The work of A. Nedi\'{c} and C.A.\ Uribe is supported by the National Science Foundation under grant no.\ CPS~15-44953. The work of A. Gasnikov is supported by Russian President grant no. MD-1320.2018.1. and RFBR no. 18-31-20005 mol\_a\_ved. \changesr{The work of A. Gasnikov and C.A. Uribe was partially supported by Yahoo! Research Faculty Engagement Program.}}%
		\thanks{\changesr{$^{*}$AR and CAU contributed equally.}}%
	}
	
	\markboth{}%
	{Optimal distributed convex optimization on slowly time-varying graphs}
	\maketitle
	
	\begin{abstract}
		We study optimal \changes{distributed} first-order optimization algorithms when \changes{the network (i.e., communication constraints between the agents) changes with time}. This problem is motivated by scenarios where agents experience network malfunctions. \changesr{Under specific constraints on the dual function, we provide a sufficient condition that guarantees a convergence rate with optimal (up to logarithmic terms) dependencies on the network and function parameters if the network changes are constrained to a \text{small} percentage $\alpha$ of the total number of iterations. We call such networks \textit{slowly} time-varying networks.} Moreover, we show that Nesterov's method has \changesr{an} iteration complexity of \changes{$\Omega\left( \left(\sqrt{\kappa_\Phi\cdot\bar\chi} + \alpha  \log(\kappa_\Phi\cdot\bar\chi)\right)\log({1}/{\eps})\right)$} for decentralized algorithms, where $\kappa_\Phi$ is condition number of the objective function, and $\bar\chi$ is a worst case bound on the condition number of the sequence of communication graphs. \changes{Additionally, we provide an explicit upper bound on $\alpha$ in terms of the condition number of the objective function and network topologies.}
	\end{abstract}
	
	
	\IEEEpeerreviewmaketitle
	
	\section{Introduction}
	\IEEEPARstart{I}{ncreasing} amounts of data and privacy constraints in distributed storage systems, as well as the distributed nature of data sources, has driven the development of distributed optimization algorithms that can be executed over networks. For example, consider the machine learning problem with a vector of parameters $y\in\RR^d$ and a loss function $L(\mathbf A, y)$, where $\mathbf A$ is a training set of $l$ samples, and each sample is a vector of $\RR^m$. Moreover, assume the dataset $\mathbf A$ is not available in the memory of a single computer due to its size and communication costs, but is divided into $n$ parts $\{\mathbf{A_i}\}_{i=1}^n$ and stored on $n$ different machines. Therefore, one seeks to take into account the information constraints induced by the distributed nature of the data.
	
	The distributed data generation and storage requires the study of the fundamental performance limits of distributed optimization algorithm that can be executed over a network~\cite{sca17,sca2018,he2018,uribe2018dual,Sun2018,li2018sharp}. One primary objective is to understand whether one can achieve the same convergence rate of centralized algorithms by using distributed methods. In \cite{Nedic2017achieving}, the authors first showed that distributed algorithms could achieve linear convergence rates when optimizing sums of strongly convex and smooth functions, in comparison with previous algorithms such as the distributed sub-gradient \cite{Nedic2009}. In \cite{Qu2017,Jakovetic}, the authors show that one can accelerate distributed algorithms and achieve convergence rate close to centralized methods. However, dependencies on the function parameter and network topology were not optimal. In \cite{sca17}, the authors proposed a dual-based approach \cite{Wu2017,Zhang2017} and provided the first result on complexity lower bounds for distributed optimization over networks for sums of strongly convex and smooth functions. Later in \cite{sca2018}, the authors extended these results to non-smooth problems or non-strongly convex problems. It was shown that distributed optimization algorithms could achieve the same convergence rates as their centralized counterparts with an additional multiplicative cost related to the communication network. Recent work has shown complexity lower bounds for distributed non-convex optimization problems as well~\cite{Sun2018}. 
	
	Existing approaches show optimal convergence rates where graphs are assumed fixed~\cite{sca17, gasnikov2018universal, lan2017communication}. We focus on the case where the network is allowed to \textit{change with time}. These changes occur, for example, due to technical malfunctions and loss of connectivity between nodes \cite{gasnikov2018universal}. As a result, the change in topology induces a change in the distributed problem formulation \changes{and leads to time-varying optimization}. These changes can come both from changing cost function and changing constraints~\cite{simonetto2014distributed, chen2017bandit}.  \changes{Time-varying problems with continuous time \cite{rahili2017distributed, fazlyab2016self_triggered}} and discrete time \cite{bernstein2018online, tang2018running} have been studied before. Also, there are distributed algorithms that can be executed over time-varying networks and achieve linear convergence rates, such as DIGing \cite{Nedic2017achieving}, Push-Pull Gradient Method \cite{Pu2018}, PANDA \cite{Maros2018} and \changes{Acc-DNGD~\cite{Qu2017}}. \changes{Nevertheless, optimal convergence rate dependencies with respect to the network and function parameters  are not yet fully understood. We provide a comparative performance analysis with such methods.} 
	
	In this paper, we study the convex optimization problem
	{\small
		\begin{align}\label{eq:initial_problem_0}
			\varphi(y) = \sum_{i=1}^{n} \varphi_i(y) \longrightarrow \min_{y\in\RR^d},
		\end{align}
	}
	where $\varphi_i:\RR^d\to\RR$ is a convex function for each $i = 1,\cdots,n$. We focus on the distributed problem where each of the functions $\varphi_i$ is privately held by a computational entity in a network. That is, each node or agent~$i$ on a network has access to $\varphi_i$ only, and yet, the group of agents seek to solve the optimization problem in~\eqref{eq:initial_problem_0} by repeated interactions with other agents following the communication constraints imposed by the network. The interactions between the agents are driven by a sequence of graphs $\{\mathcal{G}_k\}_{k=1}^{\infty}$, where $\mathcal{G}_k = (V,E_k)$ is a connected undirected graph with $V = \{1,\ldots,n\}$ and $E_k$ is a set of edges such that $(j,i)\in E_k$ if a pair of nodes $i,j\in V$ can communicate at time instant $k$, see Figure~\ref{fig:erdos_changing}.
	
	\begin{figure}[ht!]
	\centering
		\subfigure{\includegraphics[width=0.23\linewidth]{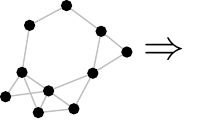}}
		\subfigure{\includegraphics[width=0.23\linewidth]{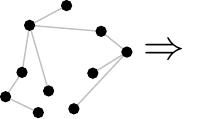}}
		\subfigure{\includegraphics[width=0.23\linewidth]{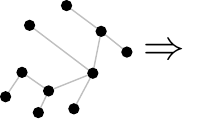}}
		\subfigure{\includegraphics[width=0.15\linewidth]{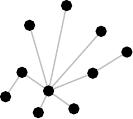}}
	\caption{A sequence of graphs with $10$ nodes each, but with different edge set at each time instant.} 
	\label{fig:erdos_changing}
\end{figure}
	
	We analyze how first-order methods behave when the objective function changes from time to time (under some restrictions to be specified later), with a special interest in distributed optimal accelerated methods~\cite{gasnikov2018universal}. \changesr{Our main contribution is  a sufficient condition that certifies that an optimal rate can be achieved, with an additional cost proportional to the number of changes in the network, expressed as a fraction of the number of iterations. Whether this condition is also necessary is left for future work.}
	
	This paper is organized as follows. Problem statement and dual formulation are described in Section~\ref{sec:problem_statement}. Section~\ref{sec:results} presents the general analysis for first-order methods with changing functions; we include the convergence rates of the gradient descent and Nesterov's fast gradient method. In Section~\ref{sec:main_results} we build upon the results of Section~\ref{sec:results} to propose an algorithm for distributed optimization over time-varying graphs and we provide its convergence rate. We compare our results with other distributed algorithms in Section \ref{sec:discussion}. In Section \ref{sec:numerical_experiments}, we provide numerical experiments to illustrate and numerically evaluate our theoretical results. Finally, some conclusive notes and future work are described in Section \ref{sec:conclusion}. 
	
	\vspace{-0.2cm}
	\section{Problem statement}\label{sec:problem_statement}
	
	In this section, initially we recall some basic definitions, then we present a formulation of the distributed optimization problem that incorporates the communication constraints induced by the network. The use of the network constraints allows for the formulation of a dual problem with a suitable structure for distributed computation. Finally, we formally pose the problem of distributed optimization over time-varying networks.
	
	\vspace{-0.3cm}
	\subsection{Preliminaries}
	This paper is focused on $\mu$-strongly convex $L$-smooth functions.
	\begin{definition}
		Let $f$ be a differentiable function on $\mathbb{X}$, where $\mathbb{X} = \RR^n$ or $\mathbb{X} = \RR^{d\times n}$. We say that $f$ $\mu$-strongly convex ($\mu > 0$) w.r.t. $\|\cdot\|$ if
		
		\vspace{-0.4cm}
		{\small
			\begin{align*}
				\forall x, y\in \mathbb{X}:\ f(y)\ge f(x) + \langle\nabla f(x), y-x\rangle + \frac{\mu}{2}\|y-x\|^2.
			\end{align*}
		}
		
		\vspace{-0.3cm}
		Moreover,  we say that $f$ $L$-smooth w.r.t. $\|\cdot\|$ if $\nabla f(x)$ is $L$-Lipschitz continuous w.r.t. to the dual norm $\|\cdot\|_*$, i.e.,
		{\small
			\begin{align*}
				\forall x, y\in\mathbb{X}:\ f(y)\le f(x) + \langle\nabla f(x), y - x\rangle + {L}/{2}\|y - x\|^2.
			\end{align*}
		}
	\end{definition}
	
	For simplicity of exposition we will present our results mainly on the $2$-norm $\|\cdot\|_2$ in $\RR^n$ and Frobenius norm $\|\cdot\|_F$ in $\RR^{d\times n}$. Note that $\|\cdot\|_{2*} = \|\cdot\|_2$ and $\|\cdot\|_{F*} = \|\cdot\|_F$. We also denote $\|\cdot\|_{op}$ the operator norm in $\RR^{n\times n}$ generated by $\|\cdot\|_2$, which is define as $
	\|A\|_{op} = \sup_{x\in\RR^n} \|Ax\|_2 / \|x\|_2$.

	\begin{definition}\label{def:scalar_product_frobenius}
		The scalar product for $X, Y\in\RR^{d\times n}$ is given by $
		\langle X, Y\rangle = \sum_{i=1}^d \sum_{j=1}^n X_{ij}Y_{ij}$ .
		The Frobenius norm is given by $
		\|X\|_F = \sqrt{\langle X, X\rangle}$.
	\end{definition}
	
	\begin{definition}\label{def:conjugate}
		Let $\mathbb{X}$ be a Euclidean space with a scalar product $\langle\cdot, \cdot\rangle$ and $\phi: \mathbb{X}\rightarrow\RR$. Then the conjugate function to $\phi$, denoted by $\phi^*$, is given by $
		\phi^*(Y) = \underset{X\in\mathbb{X}}{\sup} \big(\langle X, Y\rangle - \phi(X)\big)$, and the dual norm $\|Y\|_*$ is defined as $
		\|Y\|_* = \underset{X\in\mathbb{X}}{\sup}\{\langle X, Y\rangle: \|X\| \le 1\}$.
	\end{definition}
	
	\vspace{-0.4cm}
	\subsection{Dual problem formulation for static graphs}\label{subsec:gossip_matrix_and_problem_reformulation}
	
	Problem~\eqref{eq:initial_problem_0} can be equivalently rewritten as
	
	\vspace{-0.4cm}
	{\small
		\begin{equation}\label{eq:initial_problem_2}
			\Sum_{i=1}^n \varphi_i(y_i) \longrightarrow\min\ s.t.\ y_1=\cdots=y_n,\ y_i\in\RR^d\ \forall i \in V.
		\end{equation}
	}
	
	
	\vspace{-0.4cm}
	Additionally, the consensus constrains in~\eqref{eq:initial_problem_2} can be equivalently represented by the communication constraints imposed by the network topology. Particularly, we define the Laplacian of the graph $\mathcal{G}$ as: $[W]_{ij} = -1$ if $(i,j) \in E$, $[W]_{ij} = 	\text{deg}(i)$ if $i= j$, and $[W]_{ij} = 0$ otherwise, where $\text{deg}(i)$ is the degree of the node $i$, i.e., the number of neighbors of the node.
	
	
	If a graph $\mathcal{G}$ is undirected and connected. Then, the Laplacian matrix $W$ is symmetric and positive semidefinite. 
	\changesm{Moreover, by Perron-Frobenius theorem \cite{nikaido1972convex}, it holds that $YW = 0$ if and only if $Y\sqrt{W} = 0$ if and only if $y_1 = \cdots = y_n$}.  Therefore, for static graphs ~\eqref{eq:initial_problem_2} is equivalent to
	
	\vspace{-0.2cm}
	{\small
		\begin{equation}\label{eq:initial_problem_3}
			\Phi(Y) = \Sum_{i=1}^n \varphi_i(y_i) \longrightarrow\min_{Y\sqrt{W}=0},
		\end{equation}
	}
	
	\vspace{-0.2cm}
	\noindent where $Y = [y_1,\ldots, y_n ]\in\RR^{d\times n}$. 
	Note that $Y$ is a matrix in $\RR^{d\times n}$ consisting of local copies $y_i$ of the decision vector $y$ in the original problem \eqref{eq:initial_problem_0}. Using this equivalent constraint leads us to the dual function 
	
	\vspace{-0.4cm}
	{\small
		\begin{align}\label{eq:dual_problem_2}
			f(X)  & = \max_{Y\in\RR^{d\times n}} \left[-\langle X, Y\sqrt{W}\rangle - \Phi(Y)\right].
	\end{align}}
	
	\vspace{-0.5cm}
	\begin{remark}
		It is not necessary to restrict our attention to the Laplacian of the graph as communication matrix. We may use arbitrary positive weights and weighted degrees that follow the same sparsity structure. This is equivalent to using constraints of form $\omega_{ij}y_i=\omega_{ji}y_j$ with $\omega_{ij}>0$ instead of $y_i=y_j$. All the required properties are induced, and the rest of the analysis stays the same. On the other hand, it gives more flexibility for practical purposes, e.g., proper weight choice can induce better \changes{conditioning~\cite{boyd2004fastest}.} However, choosing specific weights is a separate question and therefore stays beyond our scope.
	\end{remark}
	
	Now that we have related the optimization problem with the network structure and the communication constraints it imposes, in this subsection, we show the connection between properties of function $\Phi(Y)$ in \eqref{eq:initial_problem_3} and the dual function $f(X)$ given in \eqref{eq:dual_problem_2}.
	
	\begin{lemma}[ Lemma~$3.1$ in \cite{Beck2014}, Proposition $12.60$ in~\cite{rockafellar2011variational}, Theorem $1$ in~\cite{nes05}, Theorem~$6$ in~\cite{kakade2009on}]\label{lem:kakade_dual_smoothness}
		Let $\phi$ be a closed convex function. Then $\phi$ is $\mu$-strongly convex w.r.t. $\|\cdot\|$ if and only if $f$ is ${1}/{\mu}$-smooth w.r.t. $\|\cdot\|_*$.
	\end{lemma}
	
	Lemma~\ref{lem:kakade_dual_smoothness} allows us to establish the relationship between strong convexity and smoothness of functions $\Phi(Y)$ in \eqref{eq:initial_problem_3} and $f(X)$ in \eqref{eq:dual_problem_2}. This relationship is formally stated in the next theorem, which extends the results of \cite{kakade2009on} on matrices.
	
	\begin{theorem}\label{th:primal_smooth}
		Let $\sigma_{\max}(W)$ be the largest eigenvalue and $\tilde\sigma_{\min}(W)$ be the least nonzero eigenvalue of $W^TW = W^2$, where $W$ is the Laplacian of the communication graph $\mathcal{G}=(V,E)$. Let $\Phi(Y)$ be $L_{\Phi}$-smooth and $\mu_{\Phi}$-strongly convex w.r.t. $\|\cdot\|_F$. Then,$f(X)$ in~\eqref{eq:dual_problem_2} is strongly convex with constant $\mu_f = \sqrt{\tilde\sigma_{\min}(W)}/L_{\Phi}$ on the subspace $(\kernel W)^{\bot}$ and smooth with constant \mbox{$L_f = \sqrt{\sigma_{\max}(W)}/\mu_{\Phi}$} on $\RR^{d\times n}$.
	\end{theorem}
	
	The proof of Theorem~\ref{th:primal_smooth} is presented in Appendix \ref{proof_smooth}.
	
	\vspace{-0.4cm}
	\subsection{Dual problem over time-varying networks}

	We are now ready to discuss the distributed optimization problem when the communication network changes with time. Particularly, we explicitly define this \textit{time-varying} setting as the case where the edge set changes. Thus, we consider a sequence of graphs $\{\mathcal{G}_k\}_{k=1}^{\infty}$, such that $\mathcal{G} = (V,E_k)$, i.e., the set of nodes remain the same but the edges might change with time. Therefore, the Laplacian matrix of the graph changes as well, which defines a sequence of graph Laplacians $\{W_k\}_{k=1}^{\infty}$. As a result, contrary to the fixed network setup, we work with a sequence of dual functions $f_k(\changesm{X})$, such that
	
	\vspace{-0.4cm}
	{\small
		\begin{align}\label{eq:dual_sequence}
			f_k(X) 
			&= \max_{Y\in\RR^{d\times n}}\left(-\big\langle X, Y \sqrt{W_k}\big\rangle - \Phi(Y)\right).
		\end{align}
	}
	
	\vspace{-0.4cm} 
	We assume that the network is connected for all $k\geq 0$. Then, all $W_k$ have the same nullspace: $\kernel(W_k) = \{y_1 = ... = y_n\} = \kernel(\sqrt{W_k})$. Consequently, the description of the constraint set changes from time to time, while the constraint set itself remains the same. 
	
	\changes{
		\begin{remark}\label{rem:why_connected}
			We impose a rather strong assumption of graphs being connected at every iteration. Such strong assumption is driven by the dual nature of the proposed algorithm, and the focus on guaranteeing optimal convergence rate dependencies on the function and network parameters. Primal-based methods allow for weaker connectivity assumptions, such as uniform connectivity, but to the best of the authors knowledge, optimal dependencies are not guaranteed~\cite{Qu2017,Maros2018,Nedic2017achieving,Pu2018}.
		\end{remark}
	}
	
	Moreover, we define 
	
	\vspace{-0.4cm}
	{\small
		\begin{align}\label{eq:def_theta_min_max}
			\theta_{\max} &= \sup_{k \geq 0}\left\lbrace  {\sigma_{\max}(W_k)}\right\rbrace, \text{ and }  
			\theta_{\min} = \inf_{k \geq 0} \left\lbrace {\tilde\sigma_{\min}(W_k)}\right\rbrace.
		\end{align}
	}
	
	\vspace{-0.4cm}
	\noindent
	Then, by Theorem \ref{th:primal_smooth}, every $f_k(X)$ is $\mu$-strongly convex on $\big(\kernel\ W\big)^\bot$ and $L$-smooth on $\RR^n$, where $\mu = {\sqrt{\theta_{\min}}}/{L_\Phi},\ L = {\sqrt{\theta_{\max}}}/{\mu_\Phi}$. \changes{We remark that what is changing with time is not the objective function $\Phi$ but the constraints representation, and thus the dual function $f_k$. Consequently all $f_k(X)$ have a common point of minimum and the same value of minimum. For clarity of exposition, we group some properties of the dual function in the following assumption, which holds for~\eqref{eq:dual_sequence} by definition.}
	
	\begin{assumption}\label{props}
		The sequence of convex functions $\{f_k(x)\}_{k=0}^{\infty}$ has the following properties:
		\begin{itemize}
			\item There is a point $x^*$ which is a common minimum for all the functions $f_k$.
			\item Every function $f_k(x)$ is $\mu$-strongly convex and $L$-smooth.
		\end{itemize}
	\end{assumption}
	
	\changesbr{
		\begin{remark}
			In Assumption \ref{props}, we require all functions $f_k$ to have a common point of minimum. This statement holds, for example, if for each dual function $f_k$ defined in \eqref{eq:dual_sequence} its minimizer $X^*_k\in\kernel W_k$, i.e. $[X_k^*]_1 = \ldots = [X_k^*]_n$, where $[X_k^*]_i$ denotes the $i$-th column of $X_k^*$. \\
			As an example, consider $\varphi(y) = \frac{1}{2}\Sum_{i=1}^n \|y - a_i\|^2$, where $\Sum_{i=1}^n a_i = 0$. Then
			\begin{align*}
				\Phi(Y) &= \frac{1}{2}\|y_i - a_i\|^2 = \frac{1}{2}\|Y - A\|^2 \\
				F_k(X) &= \underset{Y}{\argmax}(-\Phi(Y) - \langle Y, XW_k\rangle) \\
				&= \frac{1}{2}\|XW_k - A\|^2 - \frac{1}{2}\|A\|^2 \\
				X_k^* &= 0
			\end{align*}
		\end{remark}
	}
	
	\changes{
		\begin{remark}\label{rem:convexity_not_strict}
			Note that we require $f_k$ to be strongly convex not only on $\left(\kernel\ W\right)^\perp$, but on the whole $\RR^d$. This assumption is discussed in Section \ref{sec:main_results}. Assumption~\ref{props} does not mean that strong convexity constant of every $f_k(x)$ strictly equals to $\mu$. \changesm{Instead, $\mu = \underset{k}{\min}\ \mu(f_k)$, where $\mu(f_k)$ denotes the strong convexity parameter of $f_k$}. 
			Analogously, $L = \underset{k}{\max}\ L(f_k)$.
		\end{remark}
	}
	
	

	\vspace{-0.4cm}
	\section{Analysis of first-order methods on time-varying functions}\label{sec:results}
	
	In this section, we start by studying the convergence of the gradient descent and Nesterov's fast gradient method for the general case where the objective function changes with time but remains $L$-smooth and $\mu$-strongly convex on $\RR^n$. \changes{This is precisely the case of the dual function~\eqref{eq:dual_sequence}.} Later in Section~\ref{sec:main_results}, we will show that the trajectories of both methods are situated in $x_0 + \left(\kernel\ W\right)^\bot$, where $x_0$ is the initial point, and thus even if the functions are $\mu$-strongly convex only on $\left(\kernel\ W\right)^\bot$ and not on $\RR^n$ (which is the case for the dual of the distributed optimization problem) the studied methods still maintain the same convergence rates. Until now, we have been working with matrix argument $X\in\RR^{d\times n}$. For simplicity of exposition and without loss of generality, the following results are derived for the vector argument $x\in\RR^n$. 
	
	
	\vspace{-0.4cm}
	\subsection{Gradient descent}\label{subsec:gradient_descent}
	
	\changes{In this subsection, we show that convergence of gradient descent on time-varying functions is the same that on static functions. The proofs are omitted, because non-accelerated gradient descent is not the main focus of this paper. One can carry out the proof using a classical bound
		{\small
			\begin{align}\label{eq:grad_descent_step_inequality}
				f(x_{k+1})\le f(x_k) - \frac{1}{2L}\|\nabla f(x_k)\|_2^2
			\end{align}
		}
		and a result given in Theorem 2.5.11 in \cite{nesterov2013introductory}, which states that
		{\small
			\begin{align*}
				\langle\nabla f(x) - \nabla f(y), x - y\rangle
				\ge &\frac{\mu L}{\mu+L} \|x-y\|_2^2 \nonumber\\ 
				+ &\frac{1}{\mu+L}\|\nabla f(x) - \nabla f(y)\|_2^2
			\end{align*}
		}
		for $\mu$-strongly convex $L$-smooth functions.
	}
	
	\begin{theorem}\label{th:gradient_descent}
		Let $\{f_k(x_k)\}_{k=0}^\infty$ be a sequence of functions for which Assumption \ref{props} hold. Then, the sequence $\{x_k\}_{k=0}^{\infty}$ generated by the gradient descent method, i.e.,
		{\small
			\begin{align}\label{gradient_descent}
				x_{k+1} & = x_k - \frac{1}{L}\nabla f_k(x_k),
			\end{align}
		}
		has the following property:
		{\small
			\begin{align*}
				\|x_{k} - x^*\|_2 & \le \left(\frac{L-\mu}{L+\mu}\right)^k \|x_0 - x^*\|_2\quad\hbox{for all } k\ge 0.
			\end{align*}
		}
	\end{theorem}
	
	Next, we provide a Corollary that relates the convergence rate estimate in Theorem~\ref{th:gradient_descent} and the minimum number of iterations required to obtain an arbitrarily close approximation of the optimal solution of the optimization problem.
	
	\begin{corollary}\label{cor:gradient_descent}
		Let $\{f_k(x_k)\}_{k=0}^\infty$ be a sequence of functions for which Assumption~\ref{props} hold. Then, for any $\varepsilon >0$, the sequence generated by the iterations in~\eqref{gradient_descent} has the following property: for any $k \geq N+1$ it holds that $
		\|x_k - x^*\|\le\eps$,
		where $N \geq \ceil{\left(\log(({L+\mu})/({L-\mu})\right)^{-1} \log ({\|x_0 - x^*\|}/{\eps})}$.
	\end{corollary}

	\vspace{-0.5cm}
	\subsection{Nesterov fast gradient method}\label{subsec:nesterov}
	
	In this subsection, we provide a potential-based proof for the convergence of the Nesterov's fast gradient method \cite{nesterov2013introductory} for time-varying functions under Assumption~\ref{props}, i.e.,
	
	\vspace{-0.4cm}
	{\small
		\begin{subequations}\label{eq:nesterov_method}
			\begin{align}
				y_{k+1} &= x_k - \frac{1}{L}\nabla f_k(x_k), \\
				x_{k+1} &= \left( 1 + \frac{\sqrt{\kappa} - 1}{\sqrt{\kappa} + 1}\right)  y_{k+1} - \frac{\sqrt{\kappa} - 1}{\sqrt{\kappa} + 1} y_k,
			\end{align}
		\end{subequations}
	}
	
	\vspace{-0.4cm}
	\noindent
	with initial points $y_0 = x_0 $ and $\kappa = {L}/{\mu}$.
	
	We will follow the potential function proof methods presented in~\cite{bansal2017potential}. The general idea of such proof is the use of auxiliary potential function of the following form:
	\begin{align*}
		\Psi_k = a_k\cdot (f_k(x_k) - f(x^*)) + b_k\cdot \|x_k - x^*\|_2^2,
	\end{align*}
	with $a_k, b_k\ge 0$. If we denote 
	\begin{align}\label{eq:def_delta_phi}
		\Delta\Psi_k = \Psi_{k+1} - \Psi_{k} .
	\end{align}
	Then $
	\Psi_N = \Psi_0 + \sum_{k=1}^{N-1} \Delta\Psi_k
	$
	and
	{\small
		\begin{align}\label{eq:general_potential_bound}
			f_N(y_N) - f(x^*) &\le {\left( \Psi_0 + \sum_{k=1}^{N-1} \Delta\Psi_k\right) }/{a_N}.
		\end{align}
	}
	
	\vspace{-0.3cm}
	If an upper bound on $\Delta\Psi_k$ is obtained, then~\eqref{eq:general_potential_bound} shows the convergence rate for the method. 
	
	Nesterov's method is not a strict descent method. This becomes an obstacle in the time-varying case, because sudden changes of the function may happen too often so that the Nesterov method is run for too short periods of time and thus does not manage to make enough progress. However, this method's convergence can be proved if the number of function changes is finite. Next, we formally introduce the definition of a change in a sequence of functions.
	
	\begin{definition}
		Consider a sequence $\{f_k(x)\}_{k=0}^{\infty}$ of functions and let $f_n\not\equiv f_{n+1}$. Then we say that the sequence $\{f_k(x)\}_{k=0}^\infty$ of functions has a change at the moment $n$.
	\end{definition}
	
	
	\changes{We are now ready to state the main auxiliary result of this paper, that relates the convergence rate of an accelerated method on time-varying functions under Assumption~\ref{props}.}
	
	
	\changes{
		\begin{theorem}\label{th:nesterov}
			Let $N>0$ be a time horizon, and let Assumption~\ref{props} hold for a sequence of functions $\{f_k(x)\}_{k=0}^N$ with $0 \leq m \leq N$ changes. Then, the sequence generated by~\eqref{eq:nesterov_method} has the following property:
			{\small
				\begin{align*}
					f_N(y_N) - f^* \le \frac{L+\mu}{2}R^2 \kappa^m \left(1 - \frac{1}{\sqrt\kappa}\right)^N,
				\end{align*}
			}
			where $\kappa = {L}/{\mu}$ and $\|x_0 - x^*\|_2 \leq R$.
		\end{theorem}
	}
	
	Before proceeding to the proof of Theorem \ref{th:nesterov}, we provide a sequence of technical lemmas that will facilitate the analysis. 
	
	
	Following the technique for strongly convex functions described in \cite{bansal2017potential}, we introduce the following potential:
	{\small
		\begin{align}\label{eq:nesterov_potential}
			\Psi_k = (1 + \gamma)^k\cdot \left( f_k(y_k) - f^* + \frac{\mu}{2}\|z_k - x^*\|_2^2 \right),
		\end{align}
	}
	where $
	\gamma={1}/{\sqrt\kappa-1}$
	and $z_k$ will be defined shortly. 
	
	The next lemma provides an intermediate result regarding an auxiliary sequence $\{z_k\}$ that will come handy later in the proofs.
	
	\begin{lemma}\label{lem:auxiliary_x_z_nabla}
		Consider updates in \eqref{eq:nesterov_method} and define 
		
		\vspace{-0.4cm}
		{\small
			\begin{align*}
				\tau &= \frac{1}{\sqrt{\kappa}+1} \label{eq:def_kau}, \text{ and }
				z_{k+1} = \frac{1}{\tau} x_{k+1} - \frac{1-\tau}{\tau} y_{k+1}.
			\end{align*}
		}
		
		\vspace{-0.4cm}
		\noindent
		Then, $z_{k+1} = \frac{1}{1+\gamma}z_k + \frac{\gamma}{1+\gamma}x_k - \frac{\gamma}{\mu(1+\gamma)}\nabla f_k(x_k)$,
		where \mbox{ $		\gamma = \frac{1}{\sqrt\kappa - 1}$}.
	\end{lemma}
	\begin{proof}
		By the update rule for $x_{k+1}$ given in \eqref{eq:nesterov_method} and the definition of $\tau$, we have that
		
		\vspace{-0.3cm}
		{\small
			\begin{align*}
				x_{k+1} 
				&= \left( 1 + \frac{\sqrt\kappa - 1}{\sqrt\kappa + 1}\right)  y_{k+1} - \frac{\sqrt\kappa - 1}{\sqrt\kappa + 1} y_k \\ 
				&= (2 - 2\tau)y_{k+1} - (1 - 2\tau)y_k.
			\end{align*}
		}
		
		\vspace{-0.3cm}
		Moreover, by the definition of $z_{k+1}$, it follows that
		{\small
			\begin{align*}
				z_{k+1}  
				& =\frac{1}{\tau}x_{k+1} - \frac{1-\tau}{\tau}y_{k+1} \\
				& =\frac{1}{\tau}\left( (2-2\tau)y_{k+1} - (1-2\tau)y_k\right)  - \frac{1-\tau}{\tau}y_{k+1} \\
				& =\frac{1}{\tau} \left( (1-\tau)y_{k+1} - (1-2\tau)y_k\right) .
			\end{align*}
		}
		
		\vspace{-0.4cm}
		\noindent
		
		Now we use the update rule for $y_{k+1}$ given in \eqref{eq:nesterov_method} and also note that $x_k = (1-\tau)y_k + \tau z_k$:
		
		\vspace{-0.3cm}
		{\small
			\begin{align*}
				z_{k+1}  
				& =\frac{1}{\tau} \Big[ (1-\tau)(x_k - \frac{1}{L}\nabla f_k(x_k)) - \frac{1-2\tau}{1-\tau}(x_k - \tau z_k) \Big] \\
				& =\frac{1-2\tau}{1-\tau}z_k + \frac{\tau}{1-\tau}x_k - \frac{1-\tau}{L\tau}\nabla f_k(x_k)  \\
				& \overset{\circledOne}{=}\frac{\sqrt\kappa - 1}{\sqrt\kappa}z_k + \frac{1}{\sqrt\kappa}x_k - \frac{1}{\mu\sqrt\kappa}  \\
				& \overset{\circledTwo}{=}\frac{1}{1+\gamma}z_k + \frac{\gamma}{1+\gamma}x_k - \frac{\gamma}{\mu(1+\gamma)}\nabla f_k(x_k),
			\end{align*}
		}
		
		\vspace{-0.4cm}
		\noindent
		where $\circledOne$ is obtained by using the definitions of $\tau$ and $\kappa$, and $\circledTwo$ is obtained by using the definition of $\gamma$.
	\end{proof}
	
	
	\changesm{The next lemma will help us towards quantification of the maximum function value change in the sequence of time-varying functions}.
	
	\begin{lemma}\label{lem:delta}
		Define $\delta_k(x) = f_{k+1}(x) - f_k(x)$,
		and let both $f_k(x)$ and $f_{k+1}(x)$ be $\mu$-strongly convex and $L$-smooth, and have the same minimizer $x^*$. Then
		{\small
			\begin{align*}
				\delta_k(x) \le\frac{L-\mu}{\mu}(f_k(x) - f^*),
			\end{align*}
		}
		where $f^*$ is the common value of minimum for $\{f_k\}_{k=1}^{\infty}$.
	\end{lemma}
	
	\begin{proof}
		By strong convexity and smoothness obtain
		{\small
			\begin{align*}
				\frac{\mu}{2} \|x - x^*\|_2^2 \le f_k(x) - f^* \le \frac{L}{2} \|x - x^*\|_2^2.
			\end{align*}
		}
		
		\vspace{-0.3cm}
		The same holds for $f_{k+1}$.
		{\small
			\begin{align*}
				f_{k+1}(x) - f_k(x) &\le 
				\frac{L-\mu}{2} \|x-x^*\|_2^2 
				\le 
				\frac{L-\mu}{\mu} (f_k(x) - f^*).
			\end{align*}
		}
	\end{proof}
	
	Finally, the next lemma relates the upper bounds on the function values of the sequence of functions with the changes of a specific potential function.
	
	\begin{lemma}\label{lem:potential_change}
		Let $\{f_k(x)\}_{k=0}^\infty$ be a sequence of functions for which Assumption \ref{props} hold, and let $\Psi_k$ be the potential function given in \eqref{eq:nesterov_potential}. Then, it holds that
		\begin{equation}\label{bound:delta_phi}
			\Delta\Psi_k\le (1 + \gamma)^{k+1} \delta_k(y_{k+1}).
		\end{equation}
	\end{lemma}
	
	\begin{proof}
		The proof is analogous to the proof in Section $5.4$ in \cite{bansal2017potential}. We use the definitions of $\tau, z_k$ given in Lemma \ref{lem:auxiliary_x_z_nabla}. We have
		
		\vspace{-0.4cm}
		{\small
			\begin{align*}
				&\Delta\Psi_k\cdot(1+\gamma)^{-k} = (1+\gamma) \big( f(y_{k+1}) - f^* + \frac{\mu}{2}\|z_{k+1} - x^*\|_2^2 \big) \\
				&\qquad\qquad \qquad\qquad- \big( f(y_k) - f^* + \frac{\mu}{2} \|z_k - x^*\|_2^2 \big) \nonumber\\
				&\quad= (1+\gamma)\big( f_{k+1}(y_{k+1}) - f_{k+1}(x^*) \big) - \big(f_k(y_k) - f_k(x^*)\big)\\
				& \quad\quad+ \frac{\mu}{2} \Big[ (1+\gamma)\|z_{k+1} - x^*\|_2^2 - \|z_k - x^*\|_2^2 \Big]. \numberthis \label{eq:potential_change}
			\end{align*}
		}
		
		\vspace{-0.3cm}
		Note that from basic gradient step inequality given in \eqref{eq:grad_descent_step_inequality}, it follows that
		{\small
			\begin{align*}
				f_k(y_{k+1}) & \le f_k(x_k) - \frac{1}{2L}\|\nabla f_k(x_k)\|_2^2,
			\end{align*}
		}
		and by using the definition of $\delta_k$ in Lemma~\ref{lem:delta}, we have
		{\small
			\begin{align*}
				f_{k+1}(y_{k+1}) &\le f_k(x_k) - \frac{1}{2L}\|\nabla f_k(x_k)\|_2^2 + \delta_k(y_{k+1}).
			\end{align*}
		}
		
		\vspace{-0.4cm}
		Therefore the first term in \eqref{eq:potential_change} can be bounded as follows:
		\vspace{-0.4cm}
		{\small
			\begin{align*}
				& (1+\gamma)\big( f_{k+1}(y_{k+1}) - f_{k+1}(x^*)\big) - \big(f_k(y_k) - f_k(x^*) \big) \\
				&\quad\le (1+\gamma)\big( f_k(x_k) - \frac{1}{2L}\|\nabla f_k(x_k)\|_2^2 + \delta_k(y_{k+1}) - f^* \big)\\
				&\qquad - \big( f_k(y_k) - f^* \big) \\
				&\quad= f_k(x_k) - f_k(y_k) + \gamma(f_k(x_k) - f^*) \\ 
				&\qquad - (1+\gamma)\frac{\|\nabla f_k(x_k)\|_2^2}{2L} + (1+\gamma)\delta_k(y_{k+1}) \\  
				&\quad\le \langle\nabla f_k(x_k), x_k - y_k\rangle \\ 
				&\qquad + \gamma\big(\langle\nabla f_k(x_k), x_k - x^*\rangle - \frac{\mu}{2}\|x_k - x^*\|_2^2\big)\\
				&\qquad - \frac{1+\gamma}{2L}\|\nabla f_k(x_k)\|_2^2 + (1+\gamma)\delta_k(y_{k+1}) \numberthis\label{eq:potential_change_first_kerm}.
			\end{align*}
		}
		
		\vspace{-0.4cm}
		It is convenient to rewrite the above expression without references to $y_k$, by using Lemma \ref{lem:auxiliary_x_z_nabla}. Thus, by Lemma \ref{lem:auxiliary_x_z_nabla}, by using the definitions of $z_k, \gamma, \tau$ and $\kappa$, we deduce
		\begin{align*}
			z_k &= 
			\big(\frac{1}{\tau} - 1\big)(x_k - y_k) + x_k = 
			\sqrt\kappa(x_k - y_k) + x_k \\
			\gamma(z_k - x^*) &= \sqrt\kappa\gamma(x_k - y_k) + \gamma(x_k - x^*).
		\end{align*}
		
		Keeping in mind that $\sqrt\kappa\gamma = 1 + \gamma$, we obtain
		\begin{align*}
			&(x_k - y_k) + \gamma(x_k - x^*) = \\
			&\qquad \frac{1}{1 + \gamma}\cdot \Big[\gamma(z_k - x^*) + \gamma^2 (x_k - x^*)\Big].
		\end{align*}
		
		The expression on the right hand side of $\eqref{eq:potential_change_first_kerm}$ can be written as follows:
		\begin{align*}
			\label{eq:potential_change_first_kerm_2}
			&\frac{1}{1+\gamma}\langle\nabla f_k(x_k), \gamma (z_k - x^*) + \gamma^2(x_k - x^*)\rangle - \\
			&\quad \frac{\mu\gamma}{2}\|x_k - x^*\|_2^2 
			- \frac{1 + \gamma}{2L}\|\nabla f_k(x_k)\|_2^2 + (1+\gamma)\delta_k(y_{k+1}). \numberthis
		\end{align*}

		The obtained bound \eqref{eq:potential_change_first_kerm} is almost the same as (5.69) in \cite{bansal2017potential}. The only difference is the additional term \mbox{$(1+\gamma)\delta_k(y_{k+1})$}.
		\smallskip
		
		The second term in \eqref{eq:potential_change} is bounded in the same way as in \cite{bansal2017potential}. By Lemma \ref{lem:auxiliary_x_z_nabla}:
		
		\vspace{-0.4cm}
		{\small
			\begin{align*}
				&\frac{\mu}{2} \Big[ (1+\gamma)\|z_{k+1} - x^*\|_2^2 - \|z_k - x^*\|_2^2 \Big] \\
				&\quad = \frac{\mu}{2}(1+\gamma)\Big\|\frac{1}{1+\gamma}(z_k - x^*)+ \frac{\gamma}{1+\gamma}(x_k - x^*) \\
				&\qquad\qquad\qquad - \frac{\gamma}{\mu(1+\gamma)}\nabla f_k(x_k)\Big\|_2^2 - \frac{\mu}{2}\|z_k - x^*\|_2^2  \\
				&\quad = \frac{\mu}{2}\frac{1}{1+\gamma}\Big[ \|z_k - x^*\|_2^2 + \gamma^2\|x_k - x^*\|_2^2  + \frac{\gamma^2}{\mu^2}\|\nabla f_k(x_k)\|_2^2 \\
				&\qquad + 2\gamma\langle z_k - x^*, x_k - x^*\rangle - \frac{2\gamma}{\mu}\langle z_k - x^*, \nabla f_k(x_k)\rangle \\
				&\qquad - \frac{2\gamma^2}{\mu}\langle x_k - x^*, \nabla f_k(x_k)\rangle \Big] - \frac{\mu}{2}\|z_k - x^*\|_2^2 \numberthis \label{eq:potential_change_second_kerm}.
			\end{align*}
		}
		
		\vspace{-0.4cm}
		Now by adding \eqref{eq:potential_change_first_kerm_2} and \eqref{eq:potential_change_second_kerm}, we obtain a final bound on $\Delta\Psi_k$. Moreover, note that terms involving $\langle\nabla f_k(x_k), x_k - x^*\rangle$ and $\langle\nabla f_k(x_k), z_k - x^*\rangle$ cancel out.
		
		\vspace{-0.5cm}
		{\small
			\begin{align*}
				&\Delta\Psi_k (1+\gamma)^{-k}\le \left(-\frac{1+\gamma}{2L} + \frac{\gamma^2}{2\mu(1+\gamma)}\right) \|\nabla f_k(x_k)\|_2^2 \\
				&\ \ + \frac{\mu\gamma}{2}\left(\frac{\gamma}{1+\gamma} - 1\right)\|x_k - x^*\|_2^2 + \frac{\mu}{2}\left(\frac{1}{1+\gamma} - 1\right)\|z_k - x^*\|_2^2 \\
				&\qquad + \frac{\mu\gamma}{1+\gamma}\langle z_k - x^*, x_k - x^*\rangle + (1+\gamma)\delta_k(y_{k+1})  \\
				&\quad \le -\frac{\mu\gamma}{2(1+\gamma)}\big(\|x_k - x^*\|_2^2 + \|z_k - x^*\|_2^2\\
				&\qquad - 2\langle z_k - x^*, x_k - x^*\rangle\big) + (1+\gamma)\delta_k(y_{k+1}) \\
				&\quad = -\frac{\mu\gamma}{2(1+\gamma)} \|(x_k - x^*) - (z_k - x^*)\|_2^2 + (1+\gamma)\delta_k(y_{k+1})  \\
				&\quad \le (1+\gamma)\delta_k(y_{k+1}),
			\end{align*}
		}
		\noindent
		and the proof is complete.
	\end{proof}
	
	Now all the auxiliary lemmas are proved, and we move to the proof of Theorem~\ref{th:nesterov}.
	
	
	\begin{proof}[Proof of Theorem \ref{th:nesterov}]
		Lemmas \ref{lem:delta} and \ref{lem:potential_change} establish the connection between a potential change and the function residual, which enables to perform the proof by induction on the number of changes $m$.
		
		Induction base for $m=0$ holds due to smoothness of $\{f_k(x)\}_{k=0}^\infty$ and to the fact $x_0=y_0=z_0$:
		{\small
			\begin{align*}
				&\Psi_0 = f_0(y_0) - f^* + \frac{\mu}{2}\|z_0 - x^*\|_2^2 \\
				&\quad \le \frac{L}{2}\|y_0 - x^*\|_2^2 + \frac{\mu}{2}\|z_0 - x^*\|_2^2 = 
				\frac{L+\mu}{2} R^2 \\
				& f_N(y_N) - f^* \le 
				\frac{\Psi_0}{a^N} \le 
				\frac{L+\mu}{2} \frac{R^2}{(1+\gamma)^N}.
			\end{align*}
		}
		
		Let the induction hypothesis hold for $0, 1, ..., m$. By Lemma \ref{lem:delta} and using the fact that \eqref{bound:delta_phi} implies $\Delta\Psi_k\leq 0$ unless $k=n_i$ for some $i$ we get
		{\small
			\begin{align*}
				f_N&(y_N) - f^* 
				\le {\left( \Psi_0 + \sum\limits_{k=1}^{m} \Delta\Psi_{n_k}\right) }/{a_N} \\
				& \le {\left( \Psi_0 + \Sum_{k=1}^m (1+\gamma)^{n_k+1}\delta_{n_k}(y_{n_k+1})\right) }/{(1+\gamma)^N} \\
				& \le {\Psi_0 + \Sum_{k=1}^m (1+\gamma)^{n_k+1}\cdot \frac{L-\mu}{\mu}(f_{n_k}(y_{n_k+1}) - f^*)} / {(1+\gamma)^N}.
			\end{align*}
		}
		
		Since the function changes take place at the moments $n_1, ..., n_m$, the bound is true for $f_{n_1}, ..., f_{n_m}$.
		{\small
			\begin{align*}
				&(1+\gamma)^N (f_N(y_k) - f^*) \\
				&\quad \le \Psi_0 + \Sum_{k=1}^m (1+\gamma)^{n_k+1}\cdot\frac{L-\mu}{\mu}\Big(\frac{L}{\mu}\Big)^{k-1}\frac{L+\mu}{2}\frac{R^2}{(1+\gamma)^{n_k+1}}  \\
				&\quad \le\frac{L+\mu}{2} R^2 \Big( 1 + \Sum_{k=1}^m \frac{L-\mu}{\mu} \Big(\frac{L}{\mu}\Big)^{k-1} \Big)  \\
				&\quad \le\frac{L+\mu}{2} R^2 \Big(1 + \frac{L-\mu}{\mu}\cdot\frac{\Big(\frac{L}{\mu}\Big)^m - 1}{\frac{L}{\mu} - 1}\Big)  \\
				&\quad =\Big(\frac{L}{\mu}\Big)^m \frac{L+\mu}{2} R^2. \numberthis \label{eq:nesterov_bound}
			\end{align*}
		}
		
		Dividing \eqref{eq:nesterov_bound} by $(1+\gamma)^N$ finishes the proof.
	\end{proof}
	
	
	\changes{In the next section, we will use the result in Theorem~\ref{th:nesterov} for the convergence rate analysis of accelerated methods in distributed optimization over time-varying graphs. Note that the result in Theorem~\ref{th:nesterov} we set a fixed time horizon $N$ and a fixed number of changes $m$. Our main result is going to be stated for the general case where the number of changes is a fixed percentage of the number of iterations $N$. }

	
	\section{An accelerated method for distributed optimization over time-varying functions }\label{sec:main_results}
	
	In this section, we present the main result regarding the convergence rate of the distributed Nesterov fast gradient method for time-varying functions. It states that this method is linearly convergent on a slowly time-varying network. \changes{ More specifically, we show that optimal rates are guaranteed if the number of changes in the network, expressed as a fraction of the number of iterations, is bounded. We provide this explicit bound and its depedency witht the worst case condition number in the sequene of graphs.} 
	
	
	Theorems \ref{th:gradient_descent} and \ref{th:nesterov} hold for time-varying functions which are $L$-smooth and $\mu$-strongly convex on $\RR^n$. However, our initial aim was to find a common minimum of the sequence of functions defined in~\eqref{eq:dual_sequence}. In~\eqref{eq:dual_sequence}, every function $f_k(x)$ is $\mu$-strongly convex only on the subspace $\big(\kernel\ W_k\big)^\bot $ and $L$-smooth on $\RR^n$. Therefore, we need to show that the Theorems \ref{th:gradient_descent} and \ref{th:nesterov} can be generalized on strong convexity on a subspace. To do so, we show that the iterates generated by the studied algorithms are always in the space where the functions are strongly convex. 
	
	In the next lemma, we show that the gradients of the dual function are always in $ \big(\kernel\ W\big)^\bot$.
	\begin{lemma}\label{lem:gradient_in_subspace}
		Consider the function $
		f(x) = \max\limits_{y\in\RR^n}(-\big\langle x, \sqrt Wy\big\rangle - \varphi(y))$. Then, it holds that $
		\nabla f(x)\in \big(\kernel\ W\big)^\bot$.
	\end{lemma}
	
	\begin{proof}
		Initially, denote the optimal point of the inner maximization problem of the dual function as
		\[
		y(x) = \arg\max_{y\in\RR^n}\Big(-\big\langle x, \sqrt Wy\big\rangle - \varphi(y)\Big).
		\]
		Thus, by the Demianov-Danskin formula \cite{bertsekas2009convex, danskin1967theory, demianov1972introduction} it follows that $
		\nabla f(x) = -\sqrt Wy(x)$.
		Therefore, it is sufficient to show $\langle\nabla f(x), z\rangle = 0\ \forall z\in \kernel\ W$, which follows from
		\begin{align*}
			\langle\nabla f(x), z\rangle = 
			\langle -\sqrt Wy(x), z\rangle = 
			\langle -\sqrt Wz, y(x)\rangle = 
			0.
		\end{align*}
	\end{proof}
	
	In the next lemma, we show that the iterates generated by the gradient descent method and the fast gradient method are always in the space where the strong convexity of the dual function holds. This will allow us to use the results in Section~\ref{sec:results} for the specific problem of distributed optimization over time-varying graphs.
	
	\begin{lemma}\label{lemma:subspace}
		The algorithm in~\eqref{eq:nesterov_method}, with initial point $y_0=x_0$, generates sequences that are always in $x_0 + \big(\kernel\ W\big)^\bot$.
		
	\end{lemma}
	
	\begin{proof} 
		
		
		
		The proof follows by induction. Let $
		x_k - x_0\in\left(\kernel\ W\right)^\bot$ and $\ y_k - y_0 = y_k - x_0\in\big(\kernel\ W\big)^\bot$ 
		(note that it holds for $k=0$). Then, from~\eqref{eq:nesterov_method} it holds that
		{\small
			\begin{align*}
				y_{k+1} - x_0 &= (x_k - x_0) - \frac{1}{\mu}\nabla f_k(x_k) \in\left(\kernel\ W\right)^\bot \\
				x_{k+1} - x_0 &= \Big(1 + \frac{\sqrt\kappa-1}{\sqrt\kappa+1}\Big)(y_{k+1}-y_0) - \frac{\sqrt\kappa-1}{\sqrt\kappa+1}(y_k-y_0) \\
				x_{k+1} - x_0 &\in \big(\kernel\ W\big)^\bot.
			\end{align*}
		}
	\end{proof}
	
	\vspace{-1.cm}
	\subsection{Algorithm and main result}\label{sub:main_results}
	
	Now that we have shown that the general analysis in Section~\ref{sec:results} hold for the iterates generated by the studied methods in \eqref{gradient_descent} and \eqref{eq:nesterov_method}, we proceed to explicitly write the proposed accelerated distributed optimization algorithm for each of the agents in the network. Moreover, we analyze its convergence rate. 	
	
	\begin{algorithm}[ht]
		\caption{Distributed Nesterov Method}
		\begin{algorithmic}[1]
			\REQUIRE Each agent $i\in V$ locally holds $\mathbf{\varphi}_i$ and some iteration number $N$.
			\STATE{Choose $\tilde z_0^i = z_0^i$ for all $i \in V$}
			\FOR{$k=0,1,2,\cdots,N-1$}
			\STATE{$\tilde y_i(z_i^k) = \underset{y\in\RR^d}{\arg\max}\Big[\langle z_i^k, y\rangle - \varphi_i(y_i)\Big] $}
			\STATE{Send $\tilde y_i(z_i^k)$ to every neighbor and receive $\tilde y_j(z_j^k)$ from every neighbor.}
			\STATE{$\tilde z_i^{k+1} = z_i^k - \frac{1}{L} \sum\limits_{j=1}^{n}[W_k]_{ij}\tilde y_j(z_j^k) $}
			\STATE{$z_i^{k+1} = \left(1+\frac{\sqrt\kappa-1}{\sqrt\kappa+1}\right) \tilde z_i^{k+1} - \frac{\sqrt\kappa-1}{\sqrt\kappa+1} \tilde z_i^k$}
			\ENDFOR
		\end{algorithmic}
		\label{alg:distributed_nesterov}
	\end{algorithm}
	
	\changes{
		\vspace{-0.3cm}
		\begin{remark}\label{rem:lagrange_and_fenchel}
			\item
			\begin{itemize}
				\item We are working with the \textit{Lagrange} dual and running Nesterov method on it. Line $3$ of Algorithm~\ref{alg:distributed_nesterov} is a result of introducing $Z = X\sqrt{W}$.
				\item The algorithm uses $\argmax$ computation. We assume that functions $\varphi_i(\cdot)$ are dual-friendly and this operation is cheap. Relaxations of this assumption follows from~\cite{uribe2018dual}.
			\end{itemize}
		\end{remark}
	}
	
	We are now ready to state the main result of this paper, that provides the convergence rate of the distributed Nesterov fast gradient method over slowly time-varying networks.
	
	\changes{
		\begin{theorem}\label{th:distributed_nesterov}
			Let $\Phi$ be a $\mu_\Phi$-strongly convex $L_\Phi$-smooth function. Also, for $N\geq0$ let $\{\mathcal{G}_k\}_{k=1}^N$ be a sequence of undirected connected graphs with at most $\alpha N$ changes, where $\alpha\in (0,1/(\sqrt{\kappa}\log(\kappa)))$. For any $\varepsilon >0$, the output $\changesm{\tilde Z_N}$ of Algorithm~\ref{alg:distributed_nesterov} has the following property:
			\begin{align*}
				\tilde{f}_N (\changesm{\tilde Z_N}) - f^* \le \eps
			\end{align*}
			for $
			N \geq  (\sqrt{\kappa} + \alpha \log (\kappa)) \log \left(({L+\mu}){R^2}/{(2\eps)}\right)
			$, where $\tilde f(Z) = \Phi^*(-Z) = \max\{-\langle Z, Y\rangle - \Phi(Y)\}$, $\changesm{\tilde Z_N = [\tilde z_1^N,\cdots,\tilde z_n^N]}$, $R=\|X_0 - X^*\|_2$, $\kappa={L_\Phi}/{\mu_\Phi}\cdot{\sqrt{\theta_{\min}/\theta_{\max}}}$,   where $\theta_{\max}, \theta_{\min}$ are defined in $\eqref{eq:def_theta_min_max}$.
		\end{theorem}
	}
	
	\changes{
		\begin{proof}
			Given that $\Phi$ is $\mu_\Phi$-strongly convex $L_\Phi$-smooth it follows from Lemma~\ref{lem:kakade_dual_smoothness}, that $f$ (defined in~\eqref{eq:dual_sequence}) is $\mu$-strongly convex $L$-smooth, with $L={\sqrt{\theta_{\max}}}/{\mu_{\Phi}}$, and 
			$\mu=\sqrt{\theta_{\min}}/L_{\Phi}$.
			Now, from Theorem~\ref{th:nesterov} we have that
			\begin{align*}
				\tilde{f}_N (\changesm{\tilde Z_N}) - f^* & \leq \frac{L+\mu}{2} R^2 \exp(-N (1/\sqrt{\kappa}-\alpha\log \kappa )).
			\end{align*}
			Initially, note that in order to guarantee convergence we require $\alpha \in (0, 1/(\sqrt{\kappa}\log(\kappa)))$. Moreover,
			We need to find a bound on $N$ such that
			\begin{align*}
				\frac{L+\mu}{2} R^2 \left(\kappa^\alpha \left(1 - \frac{1}{\sqrt\kappa}\right)\right)^N \leq \varepsilon.
			\end{align*}
			Thus, $
			N \geq (\sqrt{\kappa} + \alpha \log (\kappa)) \log \left(({L+\mu}){R^2}/{(2\eps)}\right)$
		\end{proof}
	}

	\changes{
		In the next corollary, we provide convergence results for the primal problem.
		\begin{corollary}\label{col:nesterov_primal_convergence}
			Let $\tilde f(Z) - f^*\le \eps$. Then
			{\small
				\begin{align*}
					\Phi(\tilde Y(Z)) - \Phi^*\le 2\kappa\eps + L\|X^*\|\sqrt\frac{2\eps}{\mu},
				\end{align*}
			}
			where $\tilde Y(Z) = \underset{Y\in\RR^{d\times n}}{\argmax}\{-\langle Z, Y\rangle - \Phi(Y)\} = \left[\tilde y_1(z_i^N), \ldots, \tilde y_n(z_n^N)\right]$.
		\end{corollary}
	}
	\begin{proof}
		\changes{
			By definition of $\tilde f$, $\tilde f(Z) = f(X\sqrt W)$ for some $X$.
		}
		\changes{
			Strong convexity on $(\kernel W)^\bot$ and smoothness of $f$ yield 
			{\small
				\begin{align*}
					\frac{\mu}{2} \|X - X^*\|^2 &\le \eps, \\
					\|\nabla f(X) - \nabla f(X^*)\| &\le L_f\|X - X^*\| \le L_f\sqrt{\frac{2\eps}{\mu}}.
				\end{align*}
			}
			Now, note that $\nabla f(X) = -Y(X)W$, and trivially $
			f(X) \ge f(X^*)$, thus
			{\small
				\begin{align*}
					&-\left\langle X, Y(X)\sqrt W\right\rangle - \Phi(Y(X)) \ge \\
					&\qquad\qquad-\left\langle X^*, Y(X^*)\sqrt W\right\rangle - \Phi(Y(X^*)), \\
					&\Phi(Y(X)) - \Phi^* \le \langle X, \nabla f(X)\rangle \\ 
					&\qquad\qquad\le\left(\|X - X^*\|_2 + \|X^*\|\right)\cdot L\sqrt\frac{2\eps}{\mu} \\
					&\qquad\qquad\le\left(\sqrt\frac{2\eps}{\mu} + \|X^*\|\right)\cdot L\sqrt\frac{2\eps}{\mu}.
				\end{align*}
			}
			This completes the proof.
		}
	\end{proof}
	
	\changes{
		\begin{remark}\label{rem:upgrade_mu}
			The convergence result of Nesterov method in Theorem \ref{th:distributed_nesterov} depends on $\kappa_\Phi = L_\Phi / \underset{i}{\min}(\mu_{\varphi_i})$. This term can be reduced by changing the functions $\varphi_i$. Let strong convexity parameter of $\varphi_i$ equal to $\mu_i$ and denote $\ol{\mu} = ({1}/{n})\sum_{i=1}^n \mu_i$. Introduce $\hat \varphi_i(y) = \varphi_i(y_i) + (\ol{\mu} - \mu_i){\|y_i\|^2}/{2}$.
			{\small
				\begin{align*}
					\Phi(Y) 
					&= \Sum_{i=1}^n \varphi_i(y_i) = \Sum_{i=1}^n \left(\varphi_i(y_i) + (\ol{\mu} - \mu_i)\frac{\|y_i\|^2}{2}\right) \\
					&= \Sum_{i=1}^n \hat \varphi_i(y_i)
				\end{align*}
			}
			Now, if we work with $\hat \varphi_i$ instead of $\varphi_i$, it will result in a better condition number $\hat\kappa_\Phi = L_\Phi / \ol\mu_\Phi < L_\Phi / \underset{i}{\min}\{\mu_{\varphi_i}\} = \kappa_\Phi$.
		\end{remark}
	}

	Note that the number of steps in Theorem~\ref{th:distributed_nesterov} reaches the lower bound for decentralized methods in \cite{sca17}, which means that the Algorithm~\ref{alg:distributed_nesterov} is optimal for time-varying graphs with a finite number of changes. Moreover, since $
	\kappa = 
	\chi(W)\cdot\kappa_\Phi
	$,
	it follows that the factor $\kappa$ is proportional to $\chi(W)$, which is the communication graph condition number. The lower the graph condition number is, the better convergence rate we obtain. Note that if we used restriction $YW=0$ instead of $Y\sqrt W$ = 0, then it would be $\kappa\sim\chi(W)^2$, which would result in slower convergence.
	
	\changes{When the number of changes is not finite, but rather a percentage $\alpha$ of the total number of iterations, then, $\alpha$ needs to be upper bounded by $1/(\sqrt{\kappa}\log\kappa)$. This shows that optimal rates can only be achieved if the graph changes slowly.} \changesr{This provides only a sufficient condition, and it remains an open question whether this bound on $\alpha$ is also necessary.}

	\vspace{-0.3cm}
	\section{Discussion and comparison to other methods}\label{sec:discussion}
	
	In this section, we compare the performance of the accelerated gradient method to several distributed algorithms presented in other works. Particularly, we consider PANDA \cite{Maros2018}, DIGing \cite{Nedic2017achieving} \changes{and Nesterov method for static networks Acc-DNGD~\cite{Qu2017}}. These algorithms are designed to solve problem \eqref{eq:initial_problem_0} and are based on a mixing matrix sequence $\{V(k)\}_{k=1}^{\infty}$, which has the following properties:
	\begin{assumption}\label{as:mixing_matrix_sequence}
		1) (Decentralized property) If $i\ne j$ and edge $(i, j)\not\in E_k$, then $V(k)_{ij} = 0$;
		
		2) (Double stochasticity) $V(k) 1_n = 1_n,\ 1_n^T V(k) = 1_n^T$;
		
		3) (Joint spectrum property) There exists $B\in\mathbb{Z}$, $B > 0$, such that
		{\small
			\begin{align} 
				\delta = \sup_{k\ge B-1} \sigma_{\max} \left\{V_B(k) - \frac{1}{n}1_n 1_n^T \right\} < 1.
			\end{align}
		}
		
		Here $1_n = [1\ 1...\ 1]^T\in\RR^n$ and $V_B(k) = V(k) V(k-1)... V(k-B+1)$.
	\end{assumption}
	
	
	Following the arguments in \cite{Nedic2017achieving}, one can establish that matrices $(I_n - n^{-1} W(k))$ meets all the requirements in Assumption \ref{as:mixing_matrix_sequence} with $B = 1$.
	
	\vspace{-0.4cm}
	\subsection{Relation to DIGing}\label{subsec:diging}
	
	Let us give a lower bound on the theoretical convergence rate of the DIGing algorithm, which is linearly convergent and originally presented in \cite{Nedic2017achieving}, and compare it with the rate of accelerated gradient method obtained in Theorem \ref{th:nesterov}.
	
	\begin{assumption}\label{as:every_smooth_convex}
		Every $\varphi_i$ in problem \eqref{eq:initial_problem_0} is $\mu_i$-strongly convex and $L_i$-smooth w.r.t. $\|\cdot\|_2$.
	\end{assumption}
	
	\begin{proposition}\label{prop:diging_convergence_lower_bound}
		Under Assumptions \ref{as:mixing_matrix_sequence} and \ref{as:every_smooth_convex} the theoretical result for DIGing algorithm given in~\cite{Nedic2017achieving} does not guarantee a convergence rate faster than $O(\lambda_{0}^N)$, where $\lambda_0$ is defined as
		\begin{align*}
			\lambda_0 = 1 - {1}/({12\ol\kappa^{3/2}\sqrt{n}}).
		\end{align*}
		
		Here $\ol\kappa = {1}/{n}\sum_{i=1}^n{L_i}/{\mu_i}$ and $n$ is the number of vertices in the network graph.
	\end{proposition}
	
	The proof of Proposition \ref{prop:diging_convergence_lower_bound} is presented in Appendix \ref{app:proof_diging}.
	
	The convergence rate of Nesterov gradient method obtained in Theorem \ref{th:nesterov} is $O(\lambda_1^N)$ where
	{\small
		\begin{align}\label{eq:nesterov_lambda2}
			\lambda_1 = 1 - {1}/{\kappa^{1/2}\left({\theta_{\max}}/{\theta_{\min}}\right)^{1/4}}. 
		\end{align}
	}
	
	\vspace{-0.4cm}
	Note that $\kappa$ is the condition number of $f$ in \eqref{eq:initial_problem_0}, while $\ol\kappa$ is an average condition number of summands $\varphi_i$.
	
	Accelerated gradient method has several advantages as well as disadvantages in comparison with the DIGing algorithm.
	\begin{itemize}
		\item{
			Typically, the objective function condition number $\kappa$ is rather large, and the graph condition number $\left({\theta_{\max}}/{\theta_{\min}}\right)^{1/2}$ corresponds to the diameter of network graph \cite{Olshevsky2017} and therefore is not larger than $n$. Moreover, if we are working with a machine learning problem and the dataset is uniformly distributed between the computers in the network, then the summands $\varphi_i$ in \eqref{eq:initial_problem_0} have approximately the same condition number, i.e. $\ol\kappa\approx\kappa$. In this case, Nesterov accelerated method outperforms DIGing, since $\kappa^{1/2} \ll \ol\kappa^{3/2}$ and $\left({\theta_{\max}}/{\theta_{\min}}\right)^{1/4}\le\sqrt{n}$.
		}
		\item{
			The case where the graph remains connected at every time step corresponds to $B=1$ in DIGing. The DIGing algorithm is capable of working with an arbitrary number of changes and with graphs which do not stay connected all the time.
		}
		\item{
			Nesterov accelerated method's number of iterations grows linearly with the number of changes in the network, while the number of iterations of the DIGing algorithm does not depend on the number of changes.
		}
	\end{itemize}
	\changesr{Finally, we emphasise that to get a full comparison between the proposed algorithm and DIGing~\cite{Nedic2017achieving} one needs to revisit the analysis done in~~\cite{Nedic2017achieving} for this particular sequence of graphs, which is beyond the scope of their work. }
	
	\vspace{-0.3cm}
	\subsection{Relation to PANDA}
	
	PANDA is a linearly-convergent dual-based algorithm presented in \cite{Maros2018}.
	
	\begin{assumption}\label{as:cumulative_smooth_convex}
		Let $\varphi$ in \eqref{eq:initial_problem_0} be $L$-smooth and $\mu$-strongly convex w.r.t. $\|\cdot\|_2$.
	\end{assumption}
	
	\begin{proposition}\label{prop:panda_convergence_lower_bound}
		Let Assumptions \ref{as:mixing_matrix_sequence} and \ref{as:cumulative_smooth_convex} hold. Then the theoretical result for PANDA in \cite{Maros2018} does not guarantee a convergence rate better then $O(\lambda_0^N)$ where $\lambda_0$ is given by
		{\small
			\[
			\lambda_0 = 1 - \frac{9}{64}\frac{1}{\kappa^{3/2}} ,\label{eq:panda_lambda}
			\]
		}
		if PANDA step-size $c\in(0, \alpha]$, where $\alpha$ is defined as
		{\small
			\[
			\alpha = 
			2\sqrt{\kappa}\mu\left(\frac{\sqrt{(1-\delta^2)\kappa^{-2/3} + 8} - 8\delta}{\kappa^{-3/2} + 8}\right)^2.
			\]
		}
	\end{proposition}
	
	The proof of Proposition \ref{prop:panda_convergence_lower_bound} is provided in Appendix \ref{app:proof_panda}.
	
	One can make sure that the PANDA algorithm can work with step size $c > \alpha$. Although it is interesting to compare Nesterov accelerated method and PANDA with a bigger step size, the analysis, in this case, seems to be complicated and therefore is left for future work.
	
	Analogously to Section \ref{subsec:diging}, let us discuss advantages and disadvantages of the results of this paper in comparison with PANDA (i.e. compare $\lambda_0$ in \eqref{eq:panda_lambda} to $\lambda_1$ in \eqref{eq:nesterov_lambda2}).
	\begin{itemize}
		\item{
			If the objective function is badly conditioned, i.e. $\kappa\gg 1$, and the communication graph is well-conditioned, then Nesterov method outperforms PANDA. On the other hand, if $\kappa\ll\left({\theta_{\max}}/{\theta_{\min}}\right)^{1/4}$, PANDA converges faster.
		}
		\item{
			Analogously to DIGing, PANDA works under weaker assumptions and does not depend on the number of changes in the network.
		}
	\end{itemize}
	
	\vspace{-0.4cm}
	\subsection{Relation to Nesterov method on static network}
	\changes{
		In this section we provide a theoretical comparison between our method and Nesterov method on a static network presented in \cite{Qu2017}.
	}
	\begin{proposition}\label{prop:nesterov_static_comparison}
		\changes{
			The theoretical bound for Algorithm \ref{alg:distributed_nesterov} is better than the bound for distributed Nesterov method in \cite{Qu2017} if and only if
			{\small
				\begin{align*}
					{\left(\lambda_2(1 - \lambda_2)\right)^{3/2}}/{250}\cdot \sqrt{{\lambda_{\max}}/{\lambda_{\min}}} < 
					\left({L_\Phi}/{\mu_\Phi}\right)^{3/14},
				\end{align*}
			}
			where $\lambda_2$ is the second largest eigenvalue of $W$.
		}
	\end{proposition}
	
	The proof of Proposition \ref{prop:nesterov_static_comparison} is provided in Appendix \ref{app:proof_static_nesterov}.
	
	\vspace{-0.2cm}
	\section{Numerical experiments}\label{sec:numerical_experiments}
	
	In this section, we present simulation results for the Algorithm~\ref{alg:distributed_nesterov} for the \textit{rigde regression} (strongly convex and smooth) problem. Moreover, we compare its performance with the centralized fast gradient method~\cite{nesterov2013introductory}, DIGing~\cite{Nedic2017achieving}, \changes{\mbox{Acc-DNGD}~\cite{Qu2017},} and PANDA \cite{Maros2018}.
	
	The synthetic \textit{rigde regression}  problem is defined as
	{\small
		\begin{align}\label{eq:regression}
			\min_{z \in \mathbb{R}^m} \frac{1}{2nl}\|b - Hz\|_2^2 + \frac{1}{2}c \|z\|_2^2.
		\end{align}
	}
	Moreover, we seek to solve \eqref{eq:regression} distributedly over a network. Each entry of the data matrix $H \in \mathbb{R}^{nl\times m}$ is generated as an independent identically distributed random variable $H_{ij} \sim \mathcal{N}(0,1)$, the vector of associated values $b \in \mathbb{R}^{nl}$ is generated as a vector of random variables where $b = Hx^* + \epsilon$ for some predefined $x^* \in \mathbb{R}^{m}$ and $\epsilon \sim \mathcal{N}(0,0.1)$. The columns of the data matrix $H$ and the output vector $b$ are evenly distributed among the agents with a total of $l$ data points per agent. The regularization constant is set to $c = 0.1$. Thus, each agent has access to a subset of points such that
	{\small
		\begin{align*}
			b^T & = [\underbrace{b_1^T }_{\text{Agent} \ 1}\mid \cdots \mid \underbrace{b_n^T}_{\text{Agent}  \ n}] \quad \text{and }
			H^T  = [\underbrace{H_1^T }_{\text{Agent} \ 1}\mid \cdots \mid \underbrace{H_n^T}_{\text{Agent}  \ n}],
		\end{align*}
	}
	where $b_i \in \mathbb{R}^l$ and $H_i \in \mathbb{R}^{l\times m}$ for each agent $i\in V$. Therefore, in this setup each agent $i \in V$ has a private local function
	{\small
		\begin{align*}
			\qquad \varphi_i(x_i) \triangleq \frac{1}{2nl}\|b_i - H_ix_i\|_2^2 + \frac{1}{2}\frac{c}{n} \|x_i\|_2^2.
		\end{align*}
	}
	
	\vspace{-0.3cm}
	Moreover, the optimization problem~\ref{eq:regression} is equivalent to
	{\small
		\begin{align*}
			\min_{\sqrt{W}x =0} \sum_{i=1}^{n} \left( \frac{1}{2}\frac{1}{nl}\|b_i - H_ix_i\|_2^2 + \frac{1}{2}\frac{c}{n} \|x_i\|_2^2\right),
		\end{align*}
	}
	where $W = \bar W \otimes I_m$. 
	
	Figure~\ref{fig:results1} shows the numerical results when the network is a sequence Erd\H{o}s-R\'enyi random graphs with $100$ agents and the graph changes at: every step, every $10$ steps, and every $100$ steps. Given that Erd\H{o}s-R\'enyi random graphs condition number scales logarithmically with the number of agents, the changes do not affect the rate of convergence.  \changes{In all three cases, the performance of \mbox{Acc-DNGD} is comparable with the proposed method.} In the next examples we will see how abrupt changes can lead to the instability of the algorithm. 
	
	\begin{figure}[t]
	\centering
\subfigure{\includegraphics[width=0.8\linewidth]{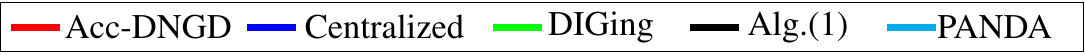}}
\subfigure{\includegraphics[width=0.45\linewidth]{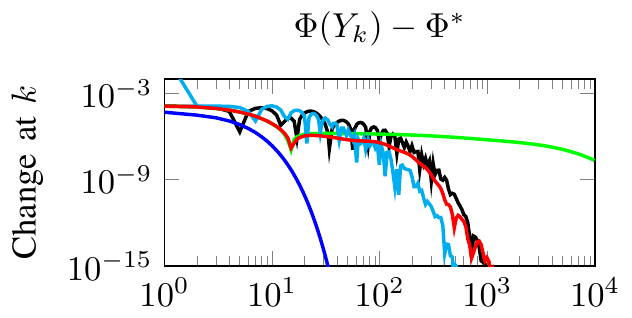}}
\subfigure{\includegraphics[width=0.45\linewidth]{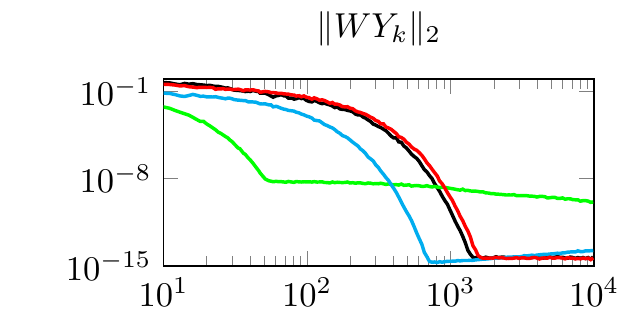}}
\subfigure{\includegraphics[width=0.45\linewidth]{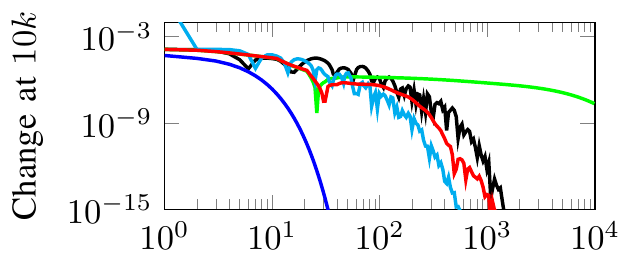}}
\subfigure{\includegraphics[width=0.45\linewidth]{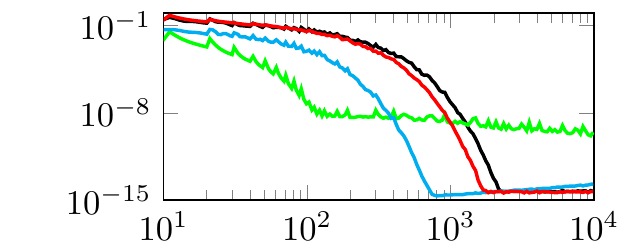}}
\subfigure{\includegraphics[width=0.45\linewidth]{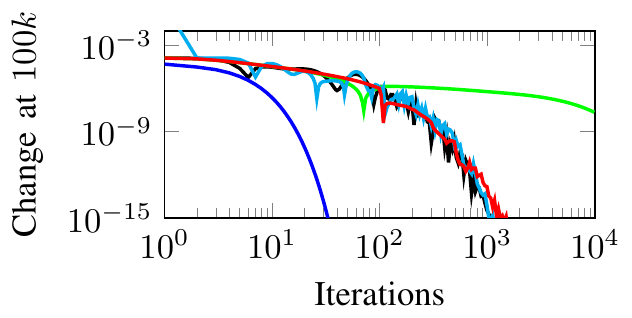}}
\subfigure{\includegraphics[width=0.45\linewidth]{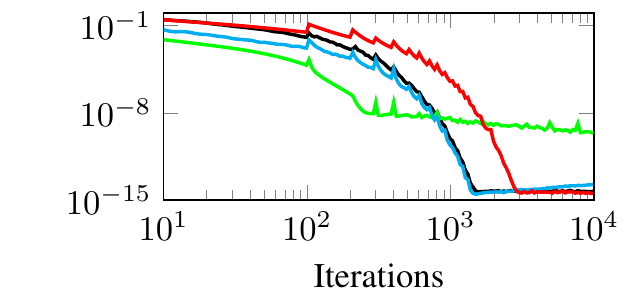}}
	\caption{Distance to optimality and distance to consensus for a sequence of Erd\H{o}s-R\'enyi random graphs with $100$ agents. Top row shows the results for graphs that change at every step, middle row shows the results for changes every $10$ steps, and bottom row shows the results for changes every $100$ steps.}
	\label{fig:results1}
\end{figure} 
	
	Figure~\ref{fig:full_line} shows the numerical results for a sequence of graphs that changes between a complete graph and a path graph every $50$, $100$ or $500$ iterations. Even if the graph changes every $50$ iterations, the convergence is maintained, due to the connectivity of the complete graph. The DIGing algorithm reaches consensus faster, but not on the optimal point, which is slower than other methods. Every time there is a change in the topology there is an increase in the distance to consensus due the the changes in the neighbor sets. When the graph changes every $500$ steps, one can see that after the initial steps as a line graph, the algorithm converges fast once the graph switches to the complete graph. For fast changing graphs PANDA has comparable performance as Alg.~\ref{alg:distributed_nesterov}, however, performance improves as the changes happen less frequently. \changes{For this rapid abrupt change in the network topology, \mbox{Acc-DNGD} convergence faster for rapid changes. However, as the changes frequency decreases, the performance of Algorithm~\ref{alg:distributed_nesterov} is comparable with \mbox{Acc-DNGD}}.
	
	Figure~\ref{fig:star_cycle} shows the numerical results for a sequence of graphs that changes between a \changesm{cycle} and a \changesm{star} graph every $50$, $100$ and $500$ iterations. If the graph changes quickly, every $100$ for this case, Alg.~\ref{alg:distributed_nesterov} diverges, i.e., the proposed accelerated method is not able to keep up with the network changes. This is evident in this case since we are switching between two graphs with relatively large condition number. It is only when the graph changes happen \changesm{rarely} enough, i.e., every $500$ steps, than the proposed method \changesm{converges}. \changes{In this third example, we observe that \mbox{Acc-DNGD} outperforms Algorithm~\ref{alg:distributed_nesterov}, at it maintains convergence even if the network changes rapidly. However, as the changes become infrequent Algorithm~\ref{alg:distributed_nesterov} shows faster convergence.}

	\begin{figure}[t]
	\centering
			\subfigure{\includegraphics[width=0.21\linewidth]{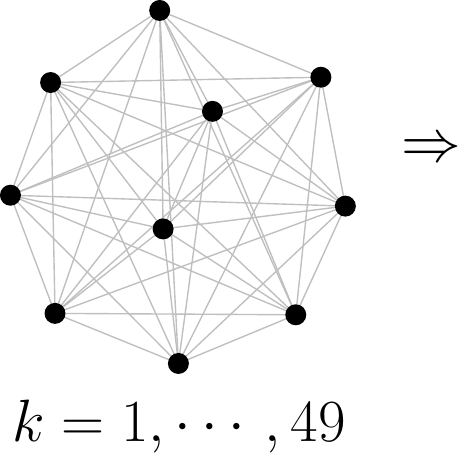}}
	\subfigure{\includegraphics[width=0.22\linewidth]{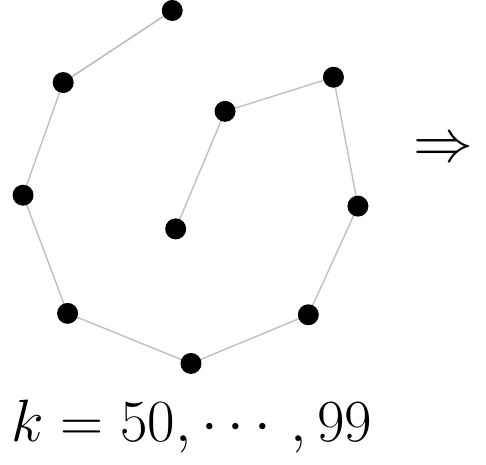}}
	\subfigure{\includegraphics[width=0.23\linewidth]{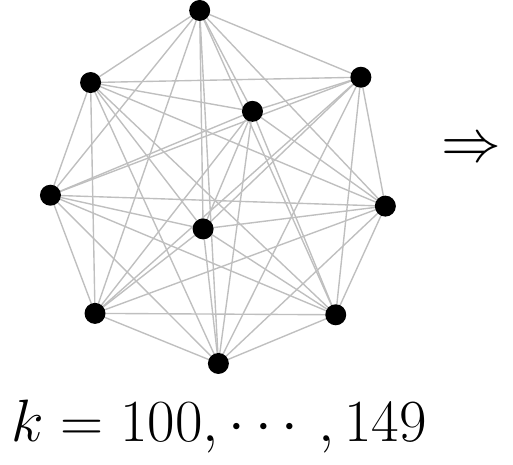}}
	\subfigure{\includegraphics[width=0.8\linewidth]{Figures/header_new}}
	\subfigure{\includegraphics[width=0.45\linewidth]{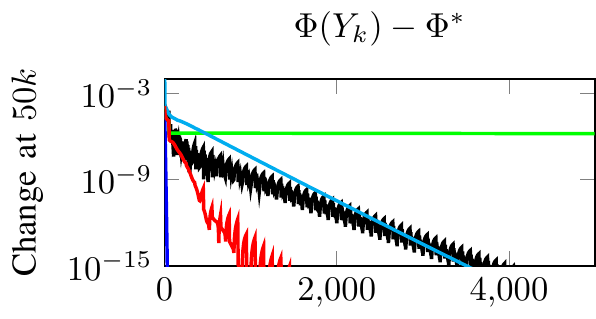}}
	\subfigure{\includegraphics[width=0.45\linewidth]{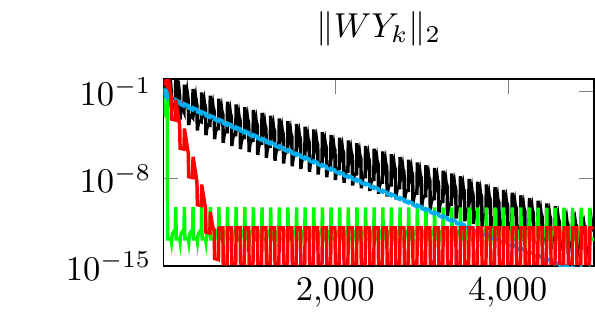}}
	\subfigure{\includegraphics[width=0.45\linewidth]{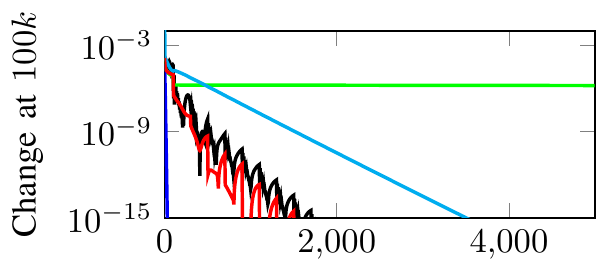}}
	\subfigure{\includegraphics[width=0.45\linewidth]{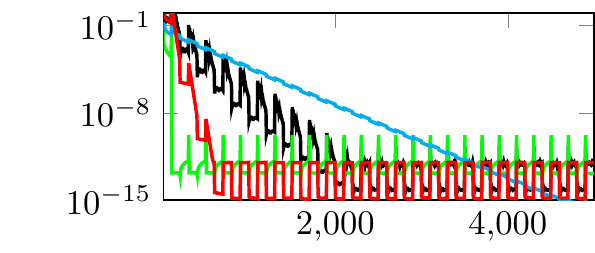}}
	\subfigure{\includegraphics[width=0.45\linewidth]{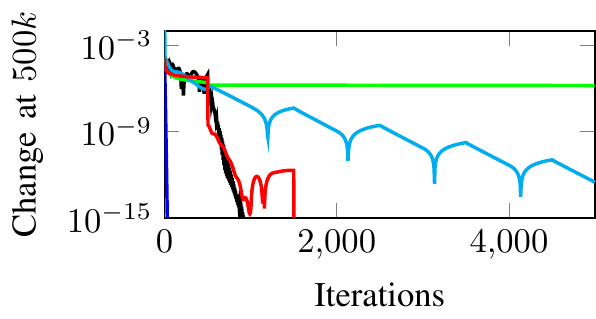}}
	\subfigure{\includegraphics[width=0.45\linewidth]{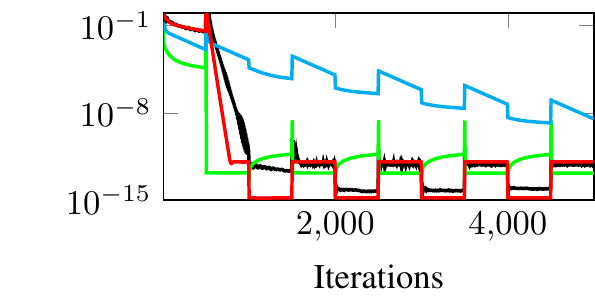}}
	\caption{Distance to optimality and distance to consensus for a network of $100$ Agents on a sequence of graphs that shuffles between a complete graph and a path graph every  $50$ iterations,  $100$ iterations, and $500$ iterations.}
\label{fig:full_line}
\end{figure} 
	\begin{figure}[t]
	\centering
			\subfigure{\includegraphics[width=0.21\linewidth]{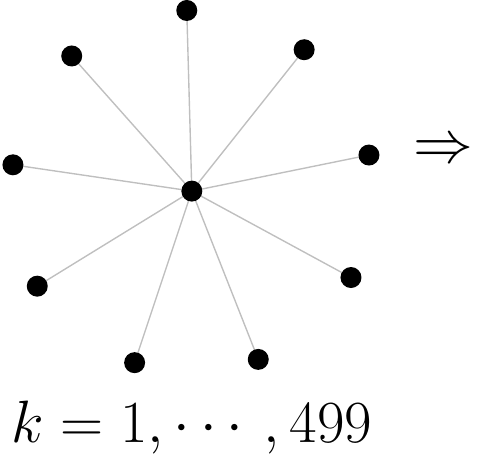}}
	\subfigure{\includegraphics[width=0.22\linewidth]{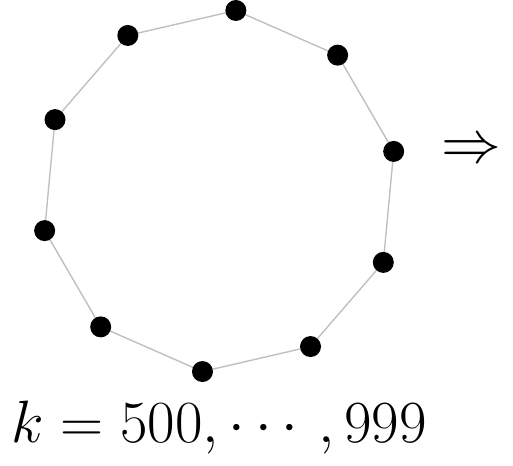}}
	\subfigure{\includegraphics[width=0.23\linewidth]{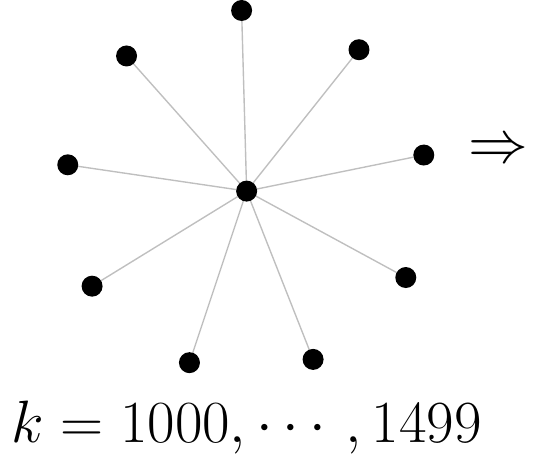}}
\subfigure{\includegraphics[width=0.8\linewidth]{Figures/header_new}}
\subfigure{\includegraphics[width=0.45\linewidth]{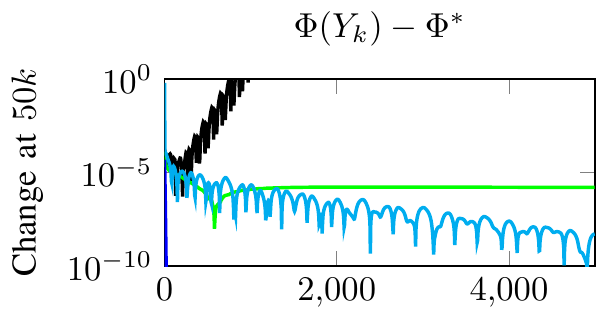}}
\subfigure{\includegraphics[width=0.45\linewidth]{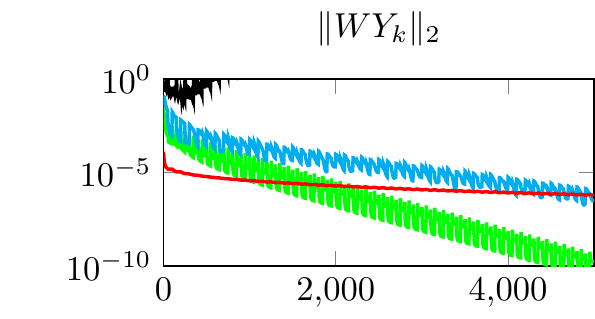}}
\subfigure{\includegraphics[width=0.45\linewidth]{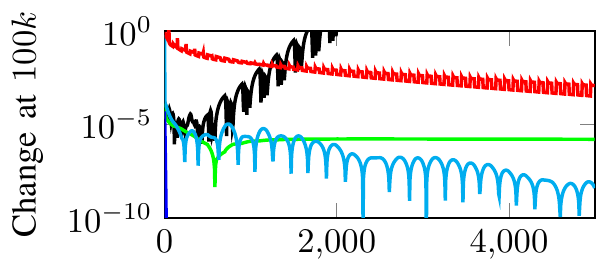}}
\subfigure{\includegraphics[width=0.45\linewidth]{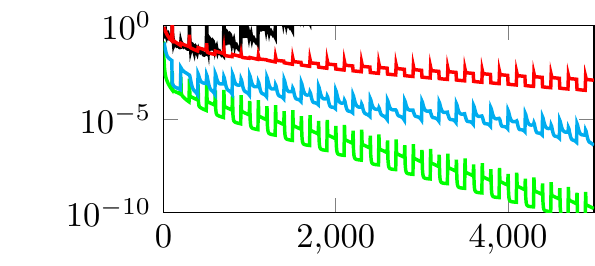}}
\subfigure{\includegraphics[width=0.45\linewidth]{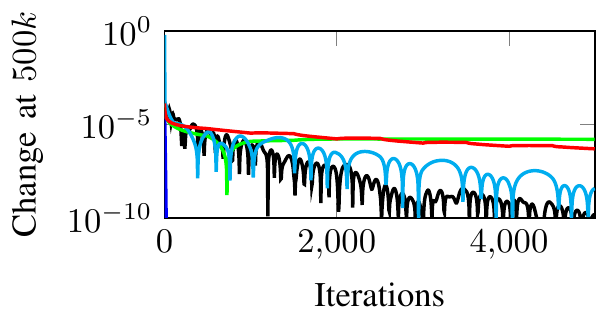}}
\subfigure{\includegraphics[width=0.45\linewidth]{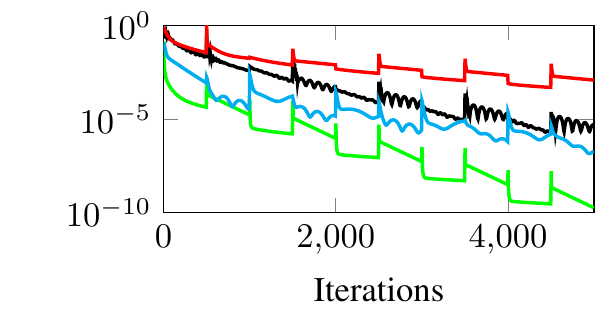}}
	\caption{Distance to optimality and distance to consensus for a network of $100$ Agents on a sequence of graphs that shuffles between a star graph and a cycle graph every $50$ iterations, $100$ iterations, and $500$ iterations.}
\label{fig:star_cycle}
\end{figure} 
	
	\changesr{Appendix~\ref{new_simulations} provides additional simulations of the proposed algorithm for the problem of regularized logistic regression for training linear classifiers. , i.e.,
		{\small
			\begin{align}\label{fun:logistic}
				\min_{x\in \mathbb{R}^m} \frac{1}{2nl} \sum_{i=1}^{nl} \log \left(1 + \exp\left(-y_i \cdot A_i^Tx\right)\right) + \frac{1}{2}c \|x\|_2^2,
			\end{align}
		}
		where $A_i \in \mathbb{R}^d$ is a data point with $y_i \in \{-1,1\}$ as its corresponding class assignment.
		We applied the proposed method to datasets from the library \textsc{LibSVM}~\cite{Chang2011}. We seek to distributedly solve the logistic regression problem over the following datasets: \textsc{a9a}, \textsc{mushrooms}, \textsc{ijcnn1} and \textsc{covtype}.
	}
	
	\section{Conclusions and future work}\label{sec:conclusion}
	
	We study the convergence of gradient descent, and Nesterov accelerated method on time-varying networks. We theoretically prove and empirically illustrate that these methods are linearly convergent under strong convexity and smoothness of the objective function and specific assumptions on the network structure. Nesterov accelerated  method performs better in terms of the objective condition number than other methods in the literature. However, the number of required iterations grows linearly with the number of changes in the network, while other algorithms' performance does not depend on how often the network changes.
	
	The rate ${1}/{(1+\gamma)^N}$ can be improved to $\prod_{i=1}^T(1+\gamma_i)^{-1}$, where $\gamma_i = {1}/{(\sqrt{\kappa_i} - 1)}$. This would be a better bound, because not all of the functions $f_k(x)$ may be badly conditioned with $\kappa={L}/{\mu}$. Also, the convergence rate of the accelerated gradient method may be improved by using restarts, which is left for future work. \changes{Other accelerating schemes such as the use of Chebyshev accelerations require further work~\cite{sca17}.} \changesr{Finally, note that we provide only a sufficient condition on the number of changes in the network under which we can guarantee an optimal convergence rate (up to logarithmic terms). Whether this is also a necessary condition remains an open question that we will address in future work. }
	
	\appendices
	\section{Proof of Theorem \ref{th:primal_smooth}}\label{proof_smooth}
	
	\begin{proof}
		The proof bases on the connection of strong convexity and smoothness of a function and its conjugate (Lemma \ref{lem:kakade_dual_smoothness}).
		
		1) First, we show that the dual norm to $\|\cdot\|_F$ is $\|\cdot\|_F$ itself.
		{\small
			\begin{align*}
				\|Y\|_{F*} = \sup_{X\in\RR^{d\times n}}\Big(\langle X, Y\rangle: \|X\|_F\le 1\Big)
			\end{align*}
		}
		Note that $\langle X, Y\rangle\le \|X\|_F \|Y\|_F \le 1\cdot \|Y\|_F$ by Cauchy-Schwarz and $\langle\frac{Y}{\|Y\|_F}, Y\rangle = \|Y\|_F$. Thus, $\|Y\|_{F*} = \|Y\|_F$.
		
		\smallskip
		
		2) Second, let $X\in\RR^{d\times n}, A\in\RR^{n\times k}, B\in\RR^{m\times d}$ and denote $\|X\|_{op}$ the operator norm of $X$ generated by $\|\cdot\|_2$ in $\RR^n$, i.e. $\|X\|_{op} = \sup_{y\in\RR^n\setminus\{0\}}\frac{\|Xy\|_2}{\|y\|_2}$, where $\|\cdot\|_2$ is the euclidean norm in $\RR^d$. 
		Then $
		\|XA\|_F\le \|X\|_{op}\cdot\|A\|_F$, and $ \|BX\|_F\le \|X\|_{op}\cdot\|B\|_F
		$.
		
		Denote $a_1, ..., a_k$ the columns of $A$.
		
		\vspace{-0.4cm}
		{\small
			\begin{align*}
				\|XA\|_F^2 
				&= \|X(a_1...a_k)\|_F^2 = 
				\|(Xa_1\ Xa_2...Xa_k)\|_F^2 \\
				&= \|Xa_1\|_F^2 + ... + \|Xa_k\|_F^2 \\
				&\le \|X\|_{op}^2\cdot\|a_1\|_2^2 + ... + \|X\|_{op}^2\cdot \|a_k\|_2^2 \\
				&= \|X_{op}\|^2\cdot \|A\|_F^2
			\end{align*}
		}
		
		The inequality $\|BX\|_F\le \|X\|_{op}\cdot\|B\|_F$ is proved analogically.
		
		\smallskip
		
		3) Third, we show the smoothness of $f(X)$.
		{\small
			\begin{align*}
				f(X) 
				&= \max_{Y\in\RR^{d\times n}}\Big[-\Psi(Y) - \langle Y, X\sqrt W\rangle\Big] \\
				&= \max_{Y\in\RR^{d\times n}}\Big[\langle Y, -X\sqrt W\rangle - \Psi(Y)\Big] = 
				\Psi^*(-X\sqrt W)
			\end{align*}
		}
		
		The function $\Psi(Y)$ is $\mu_\Psi$-strongly convex w.r.t. $\|\cdot\|_F$, and thus $\Psi^*(Z)$ is $({1}/{\mu_\Psi})$-smooth w.r.t. $\|\cdot\|_F$ by Lemma \ref{lem:kakade_dual_smoothness}. It remains to show that $f(X) = \Psi^*(-X\sqrt W)$ is $L_f = -{\sqrt{\sigma_{\max}(W)}}/{(\mu_\Psi)}$-smooth, thus
		{\small
			\begin{align*}
				df(X) &= \langle\nabla\Psi^*(-X\sqrt W), -dX\sqrt W\rangle \\
				&= \langle-\nabla\Psi^*(-X\sqrt W)\sqrt W, d\changesm{X}\rangle \\
				\nabla f(X) &= -\nabla\Psi^*(-X\sqrt W)\sqrt W \numberthis.\label{eq:dual_gradient}
			\end{align*}
		}
		
		\vspace{-0.4cm}
		Now when the gradient is computed, we explicitly show that it is Lipschitz with constant $L_f$.
		{\small
			\begin{align*}
				&\|\nabla f(X_2) - \nabla f(X_1)\|_F \\
				&\quad\le \|\sqrt W\|_{op}\cdot \|\nabla\Psi^*(-X_1 \sqrt W) - \nabla\Psi^*(-X_2 \sqrt W)\|_F  \\
				&\quad\le \|\sqrt W\|_{op}\cdot\frac{1}{\mu_\Psi} \|(X_1 - X_2)\sqrt W\|_F \\
				&\quad\changes{\le} \frac{\|\sqrt W\|_{op}^2}{\mu_\Psi}\cdot \|X_1 - X_2\|_F  = \frac{\sqrt{\sigma_{\max}(W)}}{\mu_\Psi} \|X_1 - X_2\|_F.
			\end{align*}
		}
		
		4) Finally, we prove the strong convexity of $f(X)$. It is sufficient to show
		{\small
			\begin{align*}
				f(X+dX) - f(X) \ge 
				\langle\nabla f(X), dX\rangle + \frac{\mu_f}{2} \|dX\|_F^2.
			\end{align*}
		}
		
		Keeping in mind that $f(X) = \Psi^*(-X\sqrt W)$ and $\nabla f(X) = -\nabla\Psi^*(-X\sqrt W)\sqrt W$, we obtain
		{\small
			\begin{align*}
				&f(X+dX) - f(X) \\
				&\quad= \Psi^*(-(X+dX)\sqrt W) - \Psi^*(-X\sqrt W) \\
				&\quad\ge \langle\nabla\Psi^*(-X\sqrt W), -dX \sqrt W\rangle + {1}/{(2L_\Psi)}  \|dX \sqrt W\|_F^2 \\
				&\quad=\langle -\nabla\Psi^*(-X\sqrt W)\sqrt W, dX\rangle + {1}/{(2L_\Psi)} \|dX \sqrt W\|_F^2 \\
				&\quad=\langle\nabla f(X), dX\rangle + {1}/{(2L_\Psi)} \|dX \sqrt W\|_F^2.
			\end{align*}
		}
		\changesm{Noting that $\|dX \sqrt W\|_F^2\ge \|dX\|_F^2\cdot\sqrt{\tilde\sigma_{\min}(W)}$ concludes the proof}.
		
	\end{proof}

	\section{Proof of Proposition \ref{prop:diging_convergence_lower_bound}}\label{app:proof_diging}
	
	We will need the original result for DIGing obtained in \cite{Nedic2017achieving}:
	
	\begin{proposition}\label{prop:diging_convergence}
		Let assumptions \ref{as:mixing_matrix_sequence} and \ref{as:every_smooth_convex} hold. Denote $\ol\mu = {1}/{n}\sum_{i=1}^n \mu_i,\ \ol{\kappa} = {1}/{n}\sum_{i=1}^n{L_i}/{\mu_i}$ and $J = 3\sqrt{\ol\kappa}B^2 (1 + 4\sqrt{n}\sqrt{\ol{\kappa}})$. Then the DIGing algorithm \cite{Nedic2017achieving} generates a sequence $\{x_k\}$ such that $
		\|x_N - x^*\| = O(\lambda^N)
		$, where $\lambda$ is defined as
		{\small
			\begin{align}\label{eq:diging_lambda2}
				\lambda = 
				\begin{cases}
					\sqrt[2B]{1 - \frac{\alpha\ol{\mu}}{1.5}}, &\quad \text{if}\ \alpha\in \left(0, \alpha_0\right] \\
					\sqrt[B]{\sqrt{\frac{\alpha\ol\mu J}{1.5}}, + \delta}& \quad \text{if}\ \alpha\in\left(\alpha_0, \frac{1.5(1-\delta)^2}{\ol\mu J}\right]
				\end{cases}
			\end{align}
		}
		where $\alpha_0 = {1.5\left(\sqrt{J^2 + (1-\delta^2)J} - \delta J\right)^2}/{(\ol\mu J(J+1)^2)}$.
	\end{proposition}

	\begin{proof}[Proof of Proposition \ref{prop:diging_convergence_lower_bound}]
		
		Let us consider the two cases: $\alpha\in(0, \alpha_0]$ and $\alpha\in(\alpha_0, {1.5(1-\delta)^2}/{(\ol\mu J)}]$.
		
		1) $\alpha\in(0, \alpha_0]$.
		
		\vspace{-0.4cm}
		{\small
			\begin{align*}
				&\alpha\le 
				\frac{1.5(\sqrt{J^2+J})^2}{\ol\mu J(J+1)^2} = 
				\frac{1.5}{\ol\mu (J+1)} \\
				&\lambda = 
				\sqrt{1 - \frac{\alpha\mu}{1.5}} \le 
				\sqrt{1 - \frac{\ol\mu}{1.5}\frac{1.5}{\ol\mu(J+1)}} \\
				&\ge\sqrt{1 - \frac{1}{3\ol\kappa(1 + 4\sqrt{n}\sqrt{\ol\kappa}) + 1}} \ge 
				\sqrt{1 - \frac{1}{12\ol\kappa^{3/2}\sqrt{n}}} \\
				&\ge 1 - {1}/({12\ol\kappa^{3/2}\sqrt{n}})
			\end{align*}
		}
		
		2) $\alpha\in(\alpha_0, \frac{1.5(1-\delta)^2}{\ol\mu J}]$.
		
		\vspace{-0.2cm}
		{\small
			\begin{align*}
				\lambda 
				&\ge \delta + \sqrt{\frac{\ol\mu J}{1.5}\cdot 1.5\frac{\left(\sqrt{J^2 + (1-\delta^2)J} - \delta J\right)^2}{\ol\mu J(J+1)^2}} \\
				&= \delta + \frac{\sqrt{J^2 + (1-\delta^2)J} - \delta J}{J+1} 
				= \frac{\sqrt{J^2 + (1-\delta^2)J} + \delta}{J+1} \\
				&\overset{\circledOne}{\ge} \frac{\sqrt{J^2}}{J+1} = 
				1 - \frac{1}{J+1} \ge 
				1 - \frac{1}{J}
				\ge 1 - \frac{1}{12\ol\kappa^{3/2}\sqrt{n}}, 
			\end{align*}
		}
		where $\circledOne$ is because $0\le\delta\le 1$. In both cases, $\lambda\ge\lambda_0$.
	\end{proof}

	\section{Proof of Proposition \ref{prop:panda_convergence_lower_bound}}\label{app:proof_panda}
	
	The original result for PANDA in \cite{Maros2018} states that
	\begin{proposition}\label{prop:panda_convergence}
		Under assumptions \ref{as:mixing_matrix_sequence} and \ref{as:cumulative_smooth_convex} the convergence rate of PANDA with step size $c$ is $O(\lambda^k)$, where
		{\small
			\begin{align*}
				\lambda &= \sqrt[2B]{1 - \frac{c}{2L}}, \ c\in(0, \alpha], \text{ and } \\
				\alpha &= 2\sqrt{\kappa^{-1}}\mu\left(\frac{\sqrt{(1-\delta^2)\kappa^{-2/3} + 8} - 8\delta}{\kappa^{-3/2} + 8}\right)^2
			\end{align*}
		}
	\end{proposition}
	
	\begin{proof}[Proof of Proposition \ref{prop:panda_convergence_lower_bound}]
		It suffices to show that $\lambda\ge\lambda_0$.
		
		\vspace{-0.4cm}
		{\small
			\begin{align*}
				\alpha
				&\overset{\circledOne}{\le}
				2\sqrt{\kappa^{-1}}\mu \left(\frac{\sqrt{\kappa^{-2/3} + 8}}{\kappa^{-3/2} + 8}\right)^2 \le 
				2\sqrt{\kappa^{-1}}\mu\cdot \left(\frac{\sqrt{8 + 1}}{8 + 0}\right)^2 \\
				&= \frac{9}{32}\sqrt{\kappa^{-1}}\mu \\
				\lambda 
				&\ge \sqrt{1 - \frac{\alpha}{2L}} \ge 
				\sqrt{1 - \frac{9}{32}\sqrt{\kappa^{-1}}\mu\cdot\frac{1}{2L}} = 
				\sqrt{1 - \frac{9}{64}\kappa^{-3/2}} \\ 
				&\overset{\circledTwo}{\ge} 
				1 - \frac{9}{64}\frac{1}{\kappa^{3/2}} = 
				\lambda_0.
			\end{align*}
		}
		
		Here $\circledOne$ is because $\delta\ge 0$ due to its definition in Assumption \ref{as:mixing_matrix_sequence} and $\circledTwo$ is since $\sqrt{z}\ge z$ for all $z\in[0, 1]$.
	\end{proof}	
	
	\section{\changes{Proof of Proposition \ref{prop:nesterov_static_comparison}}}\label{app:proof_static_nesterov}
	
	\begin{proposition}\cite[Theorem~$3$]{Qu2017}
		\changes{
			Nesterov method on static networks has a convergence rate of $O((1 - C({\mu_{\Phi}}/{L_{\Phi}})^{5/7})^N)$, where $C = {\left(\lambda_2(1 - \lambda_2)\right)^{3/2}}/{250}$ and $\lambda_2$ is the second largest eigenvalue of $W$.
		}
	\end{proposition}
	
	\changes{
		\begin{proof}[Proof of Proposition \ref{prop:nesterov_static_comparison}]
			Recall the notations $\kappa_{\Phi} = {L_\Phi}/{\mu_{\Phi}}$ and $\chi(W) = {\lambda_{\max}}/{\lambda_{\min}}$.
			{\small
				\begin{align*}
					\frac{\left(\lambda_2(1 - \lambda_2)\right)^{3/2}}{250}\cdot \sqrt{\frac{\lambda_{\max}}{\lambda_{\min}}} &<
					\left(\frac{L_\Phi}{\mu_\Phi}\right)^{3/14} \\
					\frac{\left(\lambda_2(1 - \lambda_2)\right)^{3/2}}{250} &<
					\kappa_{\Phi}^{3/14}\cdot \left(\chi(W)\right)^{-1/2} \\
					\frac{\left(\lambda_2(1 - \lambda_2)\right)^{3/2}}{250} \kappa_{\Phi}^{-5/7} &< 
					\left(\kappa_{\Phi}\chi(W)\right)^{-1/2} \\
					\left(1 - \frac{\left(\lambda_2(1 - \lambda_2)\right)^{3/2}}{250} \kappa_{\Phi}^{-5/7}\right)^N &>
					\left(1 - \frac{1}{\sqrt{\kappa_{\Phi}\chi(W)}}\right)^N
				\end{align*}
			}
		\end{proof}
	}
	
	\vspace{-0.5cm}
	\section{\changesm{Logistic regression simulations}}\label{new_simulations}
	
	\changesr{We assume there is a total of $nl$ data points distributed evenly among $n$ agents, where each agent holds $l$ data points. Note that each of the agents in the network will have a local function
		\vspace{-0.2cm}
		{\small
			\begin{align*}
				\varphi_i(x) & = \frac{1}{2nl} \sum_{j=1}^{l} \log \left(1 + \exp\left(-[y^i]_j \cdot [A^i]_j^Tx\right)\right) + \frac{1}{2n}c \|x\|_2^2,
			\end{align*} 
		}
		\noindent where $A^j \in \mathbb{R}^{l\times m }$ and $y^j \in \{-1,1\}^l$ are the data points held by agent $j$ and their corresponding class assignments.
	}
	
	\begin{figure}[htbp!]
	\centering
\subfigure{\includegraphics[width=0.8\linewidth]{Figures/header_new.pdf}}
\subfigure{\includegraphics[width=0.46\linewidth]{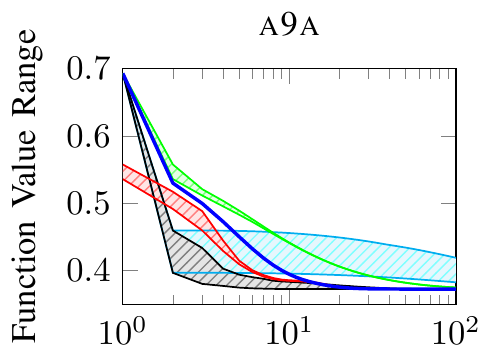}}
\subfigure{\includegraphics[width=0.46\linewidth]{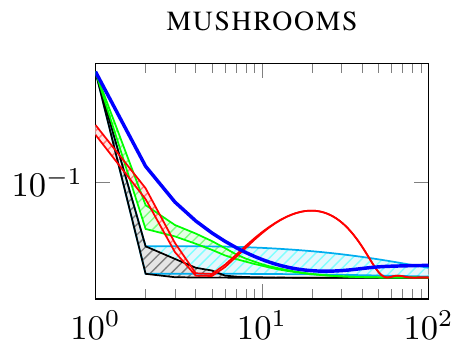}}
\subfigure{\includegraphics[width=0.46\linewidth]{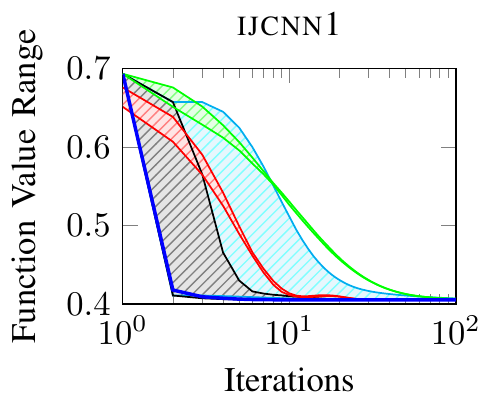}}
\subfigure{\includegraphics[width=0.46\linewidth]{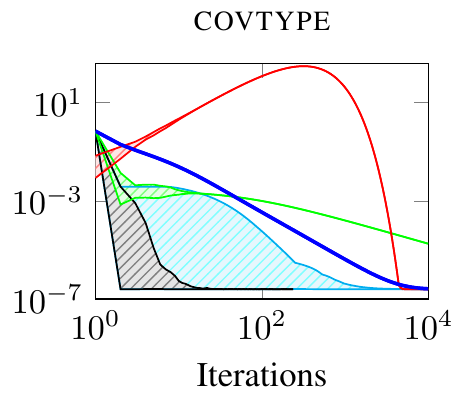}}
	\caption{Function value range achieved by the studied distributed methods for the distributed regularized logistic regression problem. The width of of each band represents the area between the maximum and minimum value achieved by the iterates held by the agents. }
	\label{fig:logistic}
\end{figure} 
	
	\changesr{Figure~\ref{fig:logistic} shows a performance comparison between the methods discussed in this paper, for problem~\eqref{fun:logistic}. We assume a network of $100$ agents, that changes every $10$ iterations, where each instance is a random geometric graph simulating a group of sensors uniformly distributed over an area of unit length, and a radius that guarantees connectivity of the network. Each agent is assigned a random sample of $100$ data points from each of the datasets. For each of the methods, and each of the datasets, we show the range of the function value among all agents. That is, the width of each band corresponds to the values achieved by the current iterates held by the agents. A narrow band represents a relatively high consensus.}
	
	\section*{acknowledgements}
	The authors thank Thinh T. Doan for his helpful comments.

	\bibliography{references,time_varying,opt_dec2}
	\bibliographystyle{IEEEtran}

\end{document}